\documentclass[a4paper, intlimits, 10pt,reqno]{amsart}
\usepackage[T1]{fontenc}
\usepackage{float}

\usepackage{times}
\usepackage[english]{babel}
\usepackage[T1]{fontenc}
\usepackage{enumitem}
\usepackage[colorlinks=true, urlcolor=blue, citecolor=blue,linkcolor=blue,
linktocpage,pdfpagelabels, bookmarksnumbered,bookmarksopen]{hyperref}
\usepackage{graphicx}

\usepackage{subcaption}
\usepackage{tabularx}
\usepackage{color}
\usepackage{eurosym}
\usepackage[numbers, sort&compress]{natbib}

\usepackage[normalem]{ulem}
\usepackage{multirow}
\usepackage{indentfirst}
\usepackage{array}
\usepackage{accents}
\usepackage{tikz}
\usetikzlibrary{arrows.meta,positioning}

\setlist[enumerate]{topsep=2pt,itemsep=3pt,parsep=1pt} 
\usepackage{tocbasic}
\makeatletter
\def\@tocline#1#2#3#4#5#6#7{\relax
  \ifnum #1>\c@tocdepth 
  \else
    \par \addpenalty\@secpenalty\addvspace{#2}%
    \begingroup \hyphenpenalty\@M
    \@ifempty{#4}{%
      \@tempdima\csname r@tocindent\number#1\endcsname\relax
    }{%
      \@tempdima#4\relax
    }%
    \parindent\z@ \leftskip#3\relax \advance\leftskip\@tempdima\relax
    \rightskip\@pnumwidth plus4em \parfillskip-\@pnumwidth
    #5\leavevmode\hskip-\@tempdima
      \ifcase #1
       \or\or \hskip 1.75em \or \hskip 2.7em \else \hskip 3.7em \fi%
      #6\nobreak\relax
    \hfill\hbox to\@pnumwidth{\@tocpagenum{#7}}\par
    \nobreak
    \endgroup
  \fi}
\makeatother

\makeatletter
\newcommand{\sectionnotoc}[1]{%
  \begingroup
    \let\@tocwrite\@gobbletwo
    \section*{#1}%
  \endgroup
}
\makeatother

\makeatletter
\newcommand{\subsectionnotoc}[1]{%
  \begingroup
    \let\@tocwrite\@gobbletwo
    \subsection*{#1}%
  \endgroup
}
\makeatother

\makeatletter
\@ifundefined{subjclassname@2020}{%
  \@namedef{subjclassname@2020}{\textup{2020} Mathematics Subject Classification}%
}{}
\makeatother

\makeatletter
\def\subsection{\@startsection{subsection}{2}%
  \z@{.5\linespacing\@plus.7\linespacing}{.3\linespacing}%
  {\normalfont\bfseries}}
\makeatother

\usepackage{amsthm, amssymb, amsfonts}
\usepackage{amsmath}

\theoremstyle{plain}
\newtheorem{theorem}{Theorem}[section]
\newtheorem{corollary}[theorem]{Corollary}

\newtheorem{lemma}[theorem]{Lemma}
\newtheorem{proposition}[theorem]{Proposition}

\newtheorem{remark}[theorem]{Remark}

\newtheorem{problem}[theorem]{Problem}

\numberwithin{equation}{section}
\theoremstyle{definition}
\newtheorem{definition}[theorem]{Definition}

\begin{document}

\title{Weak Quadruple Comparison and Structure Theory beyond Alexandrov Geometry}

\date{\today}
\thanks{}


\author{Bang-Xian Han}

\author{Liming Yin}

\begin{abstract}
  We introduce a new four-point comparison principle, called the $(\varepsilon,\delta)$-weak quadruple condition, for non-Riemannian spaces with synthetic non-negative curvature.  
  This condition holds not only for classical Alexandrov spaces with non-negative curvature, but also for many genuinely non-Riemannian spaces.
  In particular, we show that this condition is intrinsic to spaces satisfying Ohta's $S$-concavity in the full parameter range.

  Using this comparison principle, we develop a non-symmetric strainer framework and establish a Burago--Gromov--Perelman-type structure theory for finite-dimensional $S$-concave Busemann concave spaces.
  We prove that these spaces have constant integer Hausdorff dimension, satisfy the measure contraction property, are rectifiable, and admit unique Banach tangent cones almost everywhere. 
  We further show that each such space contains an open dense topological manifold part of top dimension and full measure.
  Finally, we establish Hausdorff dimension estimates for singular strata and construct natural measure-theoretic stratifications of these spaces.
  Our framework includes Alexandrov spaces with non-negative curvature as a special case, and provides tools for studying Finslerian metric spaces whose tangent cones need not be metric cones and angles need not be symmetric.

\end{abstract}

\maketitle
\tableofcontents

\section{Introduction}\label{sect:introduction}
\subsection{Toward a non-Riemannian quadruple comparison}\label{subsect:non_riemannian_quadruple}
\noindent
A central problem in synthetic lower curvature geometry is to understand how much structure can be learned from comparison principles in non-smooth spaces.
In Alexandrov geometry, this structural role is played by the \emph{quadruple comparison condition}: for every quadruple of points $(a;b,c,d)$, the three comparison angles at the base point $a$ satisfy
\begin{equation}\label{qcc}
    \tilde{\angle}_{\kappa}bac +\tilde{\angle}_{\kappa}cad+\tilde{\angle}_{\kappa}dab
    \leq 2\pi, \tag{QCC}
\end{equation}
where $\tilde{\angle}_{\kappa}$ denotes the comparison angle in the $\kappa$-plane.
This principle provides the angle control underlying the angle-based framework of Alexandrov geometry and its strainer theory, which form the foundation of Burago--Gromov--Perelman's structure theory \cite{burago1992ad}, and its subsequent refinements by Otsu--Shioya, Perelman--Petrunin, Wu, Kapovitch, Ambrosio--Bertrand, and many others \cite{otsu1994riemannian,perelman1991alexandrov,perel1994elements,perelman1994dc,perel1994extremal,petrunin1998parallel,wu1997topological,kapovitch2002regularity,kapovitch2005restrictions,ambrosio2018dc}, and have also inspired further development of structure theory of spaces with curvature bounded above, notably $\mathrm{GCBA}$ spaces, by Lytchak--Nagano--Stadler \cite{lytchak2019geod,lytchak2021topological,nagano2026wall,stadler2024cat}.

However, full quadruple comparison is too rigid beyond the Riemannian setting.
For instance, a Banach space satisfies the quadruple comparison if and only if it is a Hilbert space; see, for example, \cite[Chapter 6]{alexander2024alexandrov}.
By contrast, even in prototypical non-Riemannian spaces such as $\ell^n_p$-spaces, the sum of the Euclidean comparison angles associated with a quadruple may approach $3\pi$ as $p$ becomes large.
Thus the classical quadruple comparison \eqref{qcc} is too restrictive for the non-Riemannian setting, even though part of its geometric content, especially the comparison control it provides, remains indispensable for the strainer method.
This naturally raises the question of whether there exists an analogue of the classical quadruple comparison that is weak enough to accommodate both Alexandrov spaces with curvature bounded below and a broad class of genuinely non-Riemannian spaces with meaningful lower curvature bounds, while still retaining the ingredients needed for a structure theory.

Motivated by this question, we introduce a new notion, called the \emph{$(\varepsilon,\delta)$-weak quadruple condition} (see Definition~\ref{def:weak_quadruple_condition}).
Here $\delta$ is the \emph{straightness parameter}, and $\varepsilon$ is the \emph{angle-excess parameter}.
It is a scale-sensitive and configuration-selective comparison principle tailored to the non-Riemannian setting of \emph{Busemann spaces with synthetic non-negative curvature}.
Unlike the classical quadruple condition, it is not imposed on all
quadruples; rather, it tests only nearly-collinear quadruples with scale-dependent hierarchical configurations of points, which arise naturally in the strainer theory when almost orthogonality is recovered between strainer pairs.
This makes the condition flexible enough to include non-Riemannian models, while still strong enough to recover the strainer machinery of the classical theory.

The main conceptual point of the paper is that the role of Alexandrov quadruple comparison in the strainer method can be replaced by this weaker, configuration-selective four-point comparison principle, which arises from \emph{$S$-concavity} itself, a natural synthetic notion of non-negative curvature for non-Riemannian spaces introduced in the next section.

\smallskip

\subsection{Busemann spaces with non-negative curvature}\label{sect:intro_2}
\noindent
The study of Busemann spaces with non-negative curvature belongs to the broader program of developing synthetic curvature bounds and structure theory for spaces that need not be Riemannian.
Foundational contributions to this program are Busemann's theories of $\mathrm{G}$-spaces and Busemann convex spaces, developed in
\cite{busemann1942metric,busemann1948spaces,busemann1955geometry}.
These theories have had a lasting influence on Hilbert geometry, Finsler geometry, geometric group theory, and asymptotic geometry; see, for example, \cite{papadopoulos2014Hilbert_geometry,papadopoulos2014metric,pogorelov1998busemann,berestovskii1977finite,thurston1996Gspace,andreev2017foundations}.
More recently, Fujioka--Gu \cite{fujioka2025top} studied the topological regularity of Busemann non-positively curved spaces satisfying the local geodesic extension property, referred to as $\mathrm{GNPC}$ spaces, and extended several results from the theory of $\mathrm{GCBA}$ spaces.

While Busemann's original theory primarily concerned non-positive sectional curvature, its comparison philosophy also inspired synthetic approaches to non-negative curvature.
Conditions related to what is now called \emph{Busemann concavity} appeared in several early works, such as Kelly--Straus \cite{kelly1958curvature} on Hilbert geometry and Kann \cite{kann1960bonnet} on positively curved $\mathrm{G}$-spaces.
The modern formulation was introduced and systematically studied by Kell \cite{kell2019sectional}.
Following Kell, a geodesic space is said to be Busemann concave if, for any two constant-speed geodesics $\gamma,\eta:[0,1]\to X$ emanating from a common point, the function
\begin{equation}
  t\mapsto \frac{ \mathsf{d}\left(\gamma(t), \eta(t)\right)}{t}
\end{equation}
is non-increasing on $(0,1]$.
Kell investigated the geometry of such spaces under the assumption that they carry a nontrivial Hausdorff measure, as well as their compatibility with the measure contraction property $\mathrm{MCP}$, a weak synthetic curvature-dimension conditions introduced independently by Ohta \cite{ohta2007measure} and Sturm \cite{sturm2006geometry2}.

Another natural synthetic notion of non-negative curvature for non-Riemannian spaces is Ohta's \emph{$S$-concavity}, inspired by the work of Ball--Carlen--Lieb \cite{ball1994sharp} on uniformly smooth Banach spaces. 
It requires that, for every point $p$ and every constant-speed geodesic $\gamma$,
\begin{equation}
  \mathsf{d}\left(p,\gamma\left(\frac12\right)\right)^2
  \geq
  \frac12 \mathsf{d}(p,\gamma(0))^2
  +
  \frac12 \mathsf{d}(p,\gamma(1))^2
  -
  \frac{S}{4}\mathsf{d}(\gamma(0),\gamma(1))^2.
\end{equation}
The special case $S=1$ recovers Alexandrov spaces with non-negative curvature, whereas $S>1$ accommodates genuinely non-Riemannian examples. 
In particular, this condition is closely related, in the Finsler setting, to $2$-uniform smoothness of tangent Minkowski norms, flag curvature and Shen's tangent curvature, as well as to generalized Alexandrov--Toponogov comparison results; see \cite{ohta2009uniform,ohta2021comparison}.

Spaces satisfying both of the above synthetic non-negative curvature conditions will be called \emph{$S$-concave Busemann concave spaces}. 
This class contains a broad range of classical and genuinely non-Riemannian examples.
Indeed, Alexandrov spaces with non-negative curvature belong to this class with $S=1$.
Moreover, $2$-uniformly smooth Banach spaces with strictly convex norms, together with their closed subspaces, provide basic non-Riemannian models in this class.
In particular, for every $p\in [2,\infty)$, the spaces $\ell^n_p:=(\mathbb{R}^n,\|\cdot\|_p), \ell_p$ and $L_p$ are $S$-concave with $S=p-1$ and Busemann concave.
The class also includes connected complete Berwald spaces without conjugate points, whose flag curvature is non-negative and whose tangent Minkowski norms are $2$-uniformly smooth with uniformly bounded smoothness constants; see Section~\ref{sect:S_and_Busemann_concavity} for more details.
This broad class constitutes the principal object of the present paper, yet its finer structure remains largely unexplored.

\smallskip

\subsection{From weak quadruple comparison to structure theory}\label{sect:intro_3}
\noindent
The main obstacle to passing from Alexandrov spaces to $S$-concave Busemann concave spaces is the failure of the classical angle-based framework: there is no single canonical notion of angle.
Instead, several notions of angle arise naturally, and they need not coincide; see, for example, \cite{balestro2017angles} for different notions of angle in normed spaces.
Moreover, these notions of angle may depend on the way they are viewed, need not be symmetric, and may fail to satisfy the angular triangle inequality.
Tangent cones, in turn, need not be metric cones.
Consequently, the classical arguments underlying the Burago--Gromov--Perelman structure theory cannot be directly transferred to the present setting.

In the first paper of this series \cite{han2025structure}, we developed an angular and asymmetric strainer-theoretic framework for locally semi-convex $S$-concave Busemann concave spaces, including basic properties of tangent cones, the quasi-metric structure of spaces of directions with common length, and non-symmetric strainer coordinates. 
The present paper builds on this framework from a different perspective: it identifies the weak quadruple condition as the finite-configuration comparison principle that supplies the angle control needed in the strainer method. 
This viewpoint separates the angular formalism from the four-point comparison mechanism underlying strainer arguments.
Concretely, the weak quadruple condition enters the strainer estimates as follows.
In many arguments, one first controls a single comparison angle by elementary distance estimates, often through the Euclidean law of cosines.
The weak quadruple condition then allows the remaining comparison angles in the same strainer configuration to be controlled successively, with errors governed by its two parameters.

The key insight behind this local mechanism is that this weak comparison principle is \emph{intrinsic} to $S$-concave spaces in the full parameter range: for every prescribed angle-excess parameter, $S$-concavity guarantees the required comparison inequality for all tested quadruples with a sufficiently small straightness parameter.
The precise formulation is discussed in Section~\ref{subsect:ideas}.
This observation provides the crucial bridge between the curvature condition and the comparison principle underlying the asymmetric strainer theory.
Building on this framework, we establish a Burago--Gromov--Perelman-type structure theory for finite-dimensional $S$-concave Busemann concave spaces.

Among other results, we establish constancy of dimension, compatibility with the measure contraction property, rectifiability, and almost-everywhere uniqueness of Banach tangent cones. 
We also prove the existence of an open dense topological manifold part of full measure, establish Hausdorff dimension estimates for singular strata, and construct measure-theoretic stratifications.
Together, these results provide a comprehensive structural picture for finite-dimensional $S$-concave Busemann concave spaces.
We state these main results explicitly in the following subsection.
\smallskip

\subsection{Main results}\label{subsect:main_results}
\noindent
Throughout this subsection, $X$ denotes a complete, separable geodesic space.

Our first main result identifies various notions of dimension and establishes the compatibility of these spaces with the measure contraction property. 
We refer to Section~\ref{sect:dim} for the precise definition of constancy of dimension.
Roughly speaking, this means that all bounded open subsets have the same Hausdorff dimension.

\begin{theorem}[Proposition~\ref{prop:MCP}]\label{thm:main_thm_1}
  Let $X$ be an $S$-concave Busemann concave space with $S\geq 1$.
  Then $X$ has finite Hausdorff dimension if and only if it has finite strainer number.
  In this case, these two quantities are integer-valued and coincide with the topological dimension of $X$.
  Moreover, if $n\in \mathbb{N}$ denotes this common integer, then $X$ has constant dimension $n$, carries a nontrivial $n$-dimensional Hausdorff measure, and $(X,\mathsf d,\mathcal{H}^n)$ satisfies the measure contraction property $\mathrm{MCP}(0,n)$.
\end{theorem}

The next theorem extends some well-known geometric measure-theoretic properties of Alexandrov spaces with curvature bounded below to the setting of $S$-concave Busemann concave spaces.
\begin{theorem}[Proposition~\ref{prop:rectifiability_Banach_tangent_cone} and Corollary~\ref{cor:characterize_Banach_tangent_cone}]\label{thm:main_thm_2}
  Let $X$ be an $n$-dimensional $S$-concave Busemann concave space with $S\geq 1$.
  Then $X$ is $n$-rectifiable and $\mathcal{H}^n$-a.e.\ point admits a unique tangent cone isometric to an $n$-dimensional Banach space.
  If $X$ further satisfies local $p$-uniform convexity with $p\geq S+1$, then all tangent cones of $X$ are uniquely geodesic.
  Moreover, all Banach tangent cones are $2$-uniformly smooth and $p$-uniformly convex, with strictly convex norms.
\end{theorem}

We also obtain finer structural results for finite-dimensional $S$-concave Busemann concave spaces beyond rectifiability.

\begin{theorem}[Theorem~\ref{thm:top_manifold_part}]\label{thm:main_thm_3}
  Let $X$ be a finite-dimensional $S$-concave Busemann concave space with $S\geq 1$.
  Then $X$ contains an open dense topological manifold part of top dimension and full top-dimensional Hausdorff measure.
\end{theorem}

The next theorem provides Hausdorff dimension estimates for singular strata and yields a natural measure-theoretic stratification for finite-dimensional $S$-concave Busemann concave spaces.
\begin{theorem}[Theorem~\ref{thm:singular_set_dimension} and Corollary~\ref{cor:measure_stratification}]\label{thm:main_thm_4}
  Let $X$ be an $n$-dimensional $S$-concave Busemann concave space with $S\geq 1$.
  Then, for any $0<\delta<1/2$, the lowest-dimensional singular set $X\setminus \mathcal{A}(1,\delta)$ is discrete.
  Furthermore, for each $k=1,\ldots,n$, the Hausdorff dimension of the singular set $X\setminus \mathcal{A}(k,\delta,\delta)$ is at most $k-1$.
  In particular, $X$ admits a measure-theoretic stratification $\{X_k\}_{k=0}^n$ such that $X$ is the disjoint union of the strata $X_0,\ldots,X_n$, with $\dim_{H}(X_k)\leq k$ for each $k=0,\ldots,n$.
  Moreover, the top-dimensional stratum $X_n$ is an open dense topological $n$-manifold, while the lowest-dimensional stratum $X_0$ is discrete.
\end{theorem}
\begin{remark}
  It remains unclear whether finite-dimensional $S$-concave Busemann concave spaces admit structural results as strong as those known for finite-dimensional Alexandrov spaces with curvature bounded below.
  In the Alexandrov setting, one has a canonical \emph{topological} stratification, and the set of interior singular points has Hausdorff codimension at least 2. 
  We refer to Section~\ref{sect:problems} for further discussion.
\end{remark}
\smallskip

\subsection{Ideas and new ingredients}\label{subsect:ideas}

\noindent
The proofs are guided around the following principle: one should not try to reconstruct the full angle theory of Alexandrov geometry, since several basic mechanisms of that theory fail in the present setting. 
Instead, we isolate the angle estimates needed by the strainer method and show that it can be recovered from $S$-concavity and Busemann concavity. 

There are two basic notions of angle used throughout the paper: \emph{angles viewed from a fixed point} and \emph{angles of fixed scale}.
Several related angle-like quantities are then built from these two notions.
Each captures a different aspect of the geometry and plays a different role, since, unlike in Alexandrov geometry, there is no single canonical angle notion with all the usual properties; see Section~\ref{subsect:notions_of_angles} for details.
For the reader's convenience, we summarize them in the following table:
\begin{table}[ht]
    \centering
    \small
    \begin{tabular}{|c|c|}
        \hline
        Notation & Meaning\\
        \hline
        $\angle px\xi$ &
        angle at $x$ viewed from $p$ along the geodesic $\xi$;
        Def.~\ref{def:angle_viewed_from_fixed_point}\\
        \hline
        $\angle_x(\gamma(t),\eta(s))$ &
        angle of fixed scale $(t,s)$ between the geodesics $\gamma$ and $\eta$;
        Def.~\ref{def:angles_of_fixed_scales}\\
        \hline
        $\tilde{\angle}_x(p,q)$ or $\tilde{\angle}pxq$ &
        Euclidean comparison angle at $x$ in the comparison triangle
        $\tilde{\Delta}pxq$\\
        \hline
        $\angle_x((\gamma,l),(\eta,l))$ &
        quasi-metric on the space of directions with common length;
        Def.~\ref{def:space_of_directions_common_length}\\
        \hline
        $\angle_S(\gamma,\eta)$ &
        $S$-upper angle from $\gamma$ to $\eta$; Def.~\ref{def:S_upper_angles}\\
        \hline
    \end{tabular}
    \caption{Different notions of angle}\label{tab:angle}
\end{table}

The first new ingredient is the \emph{$(\varepsilon,\delta)$-weak quadruple condition}, introduced in Definition~\ref{def:weak_quadruple_condition}.
Its parameters are designed to provide the degrees of freedom needed to accommodate anisotropic non-Riemannian behavior, while matching the configurations produced by strainer arguments.
More precisely, $\delta$ controls the straightness and scale of the hierarchy of the quadruple configurations tested by the condition, whereas $\varepsilon$ records the allowable angle excess over the Alexandrov bound $2\pi$.
Thus, the condition does not require full quadruple comparison, but only the quantitative comparison needed in the strainer theory.
This weakening is essential, since full quadruple comparison would exclude the Finslerian examples that motivate the paper.

The verification of the weak quadruple condition is one of the core technical points of the paper.
Our key insight, Proposition \ref{lemma:verification_WQC_II}, shows that $S$-concavity itself implies the weak quadruple condition with an arbitrary prescribed angle-excess parameter, provided the tested quadruples are sufficiently close to the nearly-collinear, hierarchically arranged configurations with small scales.
More precisely, given any angle-excess parameter $\varepsilon>0$, every $S$-concave space satisfies the $(\varepsilon,\delta)$-weak quadruple condition for sufficiently small $\delta>0$ depending only on $\varepsilon$ and $S$. 
The proof uses the monotonicity encoded by $S$-concavity to obtain a relaxed triangle inequality for angles viewed from a point, in which these angles are bounded by the corresponding weak upper angles, called \emph{$S$-upper angles}, up to an error term controlled by $S$ and $\delta$. 
This estimate is then converted into one for Euclidean comparison angles.
This is the only place where the comparison principle is derived from the curvature assumption; afterwards, it is used as a replacement for Alexandrov quadruple comparison.

The second new ingredient is the strengthened strainer formalism.
As in our previous work \cite{han2025structure}, strainers are defined through an inductive procedure with additional distance control, in order to handle the asymmetry and lack of monotonicity of angles viewed from a point.
To make this formalism compatible with the weak quadruple condition, we split the single parameter of strainers into two separate ones, controlling straightness and almost orthogonality, respectively.
This separation provides the flexibility needed in the self-improvement procedure.
We also impose an additional angle requirement between strainer pairs.
Together with the weak quadruple condition, this requirement makes it possible
to recover all almost orthogonality estimates between strainer pairs successively from the estimate of one comparison angle.
This mechanism is used repeatedly in the technical core of the paper; see, for example, Proposition~\ref{prop:bi_Lipschitz}.


The third new ingredient is a strengthened dequeuing-and-enqueuing argument for the self-improvement of strainers, which improved strainers from less orthogonal ones.
In the Alexandrov setting, the straightening strategy can be applied directly to strainers, thanks to the well-definedness and symmetry of angles.
This strategy, however, is not applicable in the Busemann setting because of the asymmetric nature of strainers; in \cite{han2025structure}, this difficulty was overcome by a dequeuing-and-enqueuing procedure.
In the present setting, an additional difficulty arises: the weak quadruple condition can improve a given strainer only up to the chosen angle-excess parameter.
We therefore combine the strengthened strainer formalism above with weak quadruple conditions of varying parameters. 
This produces improved strainer pairs while keeping the entire strainer under control during the iteration.
This yields Lemma~\ref{lemma:almost_self_improvement}, which is the key ingredient for the dimension result and the measure contraction property.


Once these steps are established, the remaining structural results follow strategies similar to those developed in \cite{han2025structure}; see \cite[Section 1.3]{han2025structure} for a detailed explanation.
The principal logical dependencies are:
\medskip
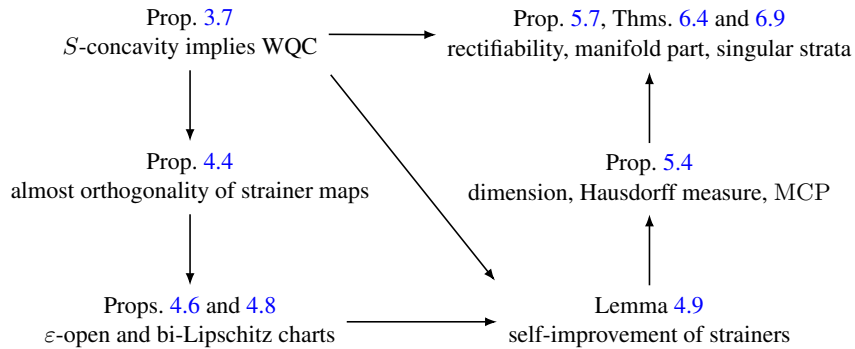
\begin{center}
\begin{tikzpicture}[
  node distance=1.15cm and 2.2cm,
  box/.style={align=center, inner sep=1pt, font=\small},
  arr/.style={-{Latex[length=1.6mm]}, semithick}
]
\node[box] (A) {
  Prop.~\ref{lemma:verification_WQC_II}\\
  $S$-concavity implies WQC
};

\node[box, below=of A] (B) {
  Prop.~\ref{prop:almost_orthogonality}\\
  almost orthogonality of strainer maps
};

\node[box, below=of B] (C) {
  Props.~\ref{prop:epsilon_openness_strainer_maps} and \ref{prop:bi_Lipschitz}\\
  $\varepsilon$-open and bi-Lipschitz charts
};

\node[box, right=of C] (D) {
  Lemma~\ref{lemma:almost_self_improvement}\\
  self-improvement of strainers
};

\node[box, above=of D] (E) {
  Prop.~\ref{prop:MCP}\\
  dimension, Hausdorff measure, $\mathrm{MCP}$
};

\node[box, above=of E] (F) {
  Prop.~\ref{prop:rectifiability_Banach_tangent_cone},
  Thms.~\ref{thm:top_manifold_part} and \ref{thm:singular_set_dimension}\\
  rectifiability, manifold part, singular strata
};

\draw[arr, shorten >=3pt, shorten <=3pt] (A) -- (B);
\draw[arr, shorten >=3pt, shorten <=3pt] (B) -- (C);
\draw[arr, shorten >=3pt, shorten <=3pt] (C) -- (D);
\draw[arr, shorten >=3pt, shorten <=3pt] (D) -- (E);
\draw[arr, shorten >=3pt, shorten <=3pt] (E) -- (F);

\draw[arr, shorten >=3pt, shorten <=3pt] (A) -- (F);
\draw[arr, shorten >=6pt, shorten <=6pt] (A.south east) to (D.north west);
\end{tikzpicture}
\end{center}
This diagram records only the main dependencies.
Its purpose is to indicate where the weak quadruple condition enters and how it is used in the strainer theory and in the global structure results.

\medskip

\noindent
\textbf{Organization}. 
Section~\ref{sect:preliminaries} recalls the basic terminology, notation, and results from \cite{han2025structure} used in the sequel.
Section~\ref{sect:weak_quadruple_condition} introduces the weak quadruple condition and establishes its compatibility with $S$-concavity.
Section~\ref{sect:strainers} develops the theory of strainers and strainer maps, including almost orthogonality, $\varepsilon$-openness, bi-Lipschitz property, and self-improvement.
Section~\ref{sect:geometric_measure_structure} introduces strainer numbers and proves Theorems~\ref{thm:main_thm_1} and~\ref{thm:main_thm_2}.
Section~\ref{sect:manifold_singular_structures} proves the results on manifold structure and singular strata, namely Theorems~\ref{thm:main_thm_3} and~\ref{thm:main_thm_4}.
Section~\ref{sect:problems} discusses open problems and further directions.

\section{Preliminaries}\label{sect:preliminaries}
\noindent
In this section, we briefly recall the basic terminology, notation, and results used throughout the manuscript.
For further background on metric geometry, we refer to \cite{burago2001course,alexander2024alexandrov}; for more details on $S$-concave and Busemann concave spaces, see \cite{ohta2021comparison,kell2019sectional}.
Our conventions are consistent with those in our previous work \cite{han2025structure}.

\medskip

\subsection{Basics from metric geometry}\label{subsect:basics}
\noindent
We denote by $(X,\mathsf d)$ a complete, separable metric space with the distance function $\mathsf d$.
Given two points $x, y \in X$, we denote by $|xy|$ or $|x,y|$ the distance $\mathsf d(x,y)$.
For $r > 0$, we denote by $B(x,r)$ and $\bar{B}(x,r)$ the open and closed balls centered at $x$ with radius $r$, respectively.

Given a subset $E \subset X$ and $r > 0$, we say that $E$ is \emph{$r$-separated} if every pair of distinct points $x, y \in E$ satisfies $|xy| \geq r$.
An $r$-separated subset $E$ is called \emph{maximal} if there does not exist any $r$-separated subset $E' \subset X$ that properly contains $E$.
For a set $E$, we denote by $\beta_E(r) \in \mathbb{N} \cup \{\infty\}$ the largest possible cardinality of a maximal $r$-separated subset of $E$.
A metric space $(X, \mathsf d)$ is said to be \emph{doubling}, or \emph{$N$-doubling}, if for any $r > 0$, the cardinality of any maximal $r/2$-separated subset of any ball $B(x, r)$ is at most $N$, i.e., $\beta_{B(x, r)}(r/2) \leq N$ for all $x \in X$ and $r > 0$.
Equivalently, any ball $B(x, r) \subset X$ can be covered by at most $N$ balls of radius $r/2$.

For a curve $\gamma \subset X$, we denote its length by $l(\gamma)$.
A curve $\gamma:[0,1] \to X$ is called a \emph{constant-speed geodesic} from $x$ to $y$ if it is a length-minimizing curve connecting $x$ and $y$ satisfying that $|\gamma(t)\gamma(s)| = |t-s|\,|xy|$ for all $t, s \in [0,1]$.
A length-minimizing curve is called a \emph{unit-speed geodesic} or \emph{geodesic} for short, if it is parametrized by arc-length.
A metric space $(X,\mathsf d)$ is said to be \emph{geodesic} if any pair of distinct points can be connected by a geodesic.
A geodesic space $(X,\mathsf d)$ is said to be non-branching if for any pair of constant-speed geodesics $\gamma,\eta:[0,1]\to X$, the condition $\gamma|_{[0,t]}=\eta|_{[0,t]}$ for some $t\in (0,1)$ implies that $\gamma=\eta$.
A function $f:X\to \mathbb{R}$ is called convex if $f\circ \gamma$ is convex on $[0,1]$ for any constant-speed geodesic $\gamma$.
A map $F:X\to Y$ between two metric spaces $(X,\mathsf d_X)$ and $(Y,\mathsf d_Y)$ is called \emph{Lipschitz} or \emph{$L$-Lipschitz} if there is $L>0$ such that $|F(x)F(y)|\leq L|xy|$ for any $x,y\in X$\footnote{For different metric spaces, we use the same notation $|xy|$ and $|F(x)F(y)|$ to denote the distances $\mathsf d_X(x,y)$ and $\mathsf d_Y(F(x),F(y))$, respectively, if there is no ambiguity.}.
It is called \emph{$L$-bi-Lipschitz} for some $L\geq 1$ if $|xy|/L\leq |F(x)F(y)|\leq L|xy|$ for any $x,y\in X$.

For a geodesic space $(X,\mathsf d)$, we denote by $\Delta xyz$ a \emph{geodesic triangle} in $X$ with vertices $x,y,z\in X$, given by the union of three geodesics joining the vertices pairwise.
A triangle $\tilde{\Delta}xyz \subset \mathbb{R}^2$ is said to be a Euclidean comparison triangle of $\Delta xyz\subset X$ if $|xy|=|\tilde{x}\tilde{y}|$, $|yz|=|\tilde{y}\tilde{z}|$ and $|xz|=|\tilde{x}\tilde{z}|$.
We denote the angle at vertex $\tilde{x}$ of the Euclidean comparison triangle $\tilde{\Delta}xyz$ by $\tilde{\angle}_x(y,z)$, or $\tilde{\angle}yxz$ for simplicity.

Given $(X,\mathsf d)$ and $\alpha\in [0,\infty)$, we denote by $\mathcal{H}^{\alpha}$ the $\alpha$-dimensional Hausdorff measure on $X$.
For a subset $E\subset X$, we denote by $\dim_{H}(E), \dim_{r}(E), \dim_{T}(E)$ the {\em Hausdorff dimension}, {\em rough dimension}, and {\em topological dimension} of $E$, respectively.
Here the rough dimension $\dim_{r}(E)$ is defined by
\begin{equation*}
  \dim_{r}(E):=
  \inf\left\{\alpha>0:\limsup_{r\searrow 0}r^{\alpha}\beta_E(r)=0\right\},
\end{equation*}
with value $\infty$ if no such $\alpha$ exists.
It is known that these notions of dimension satisfy the following inequalities:
\begin{equation}\label{eq:relation_different_dims}
  \dim_{T}(E)\leq \dim_{H}(E) \leq \dim_{r}(E),\quad \text{for any }E\subset X.
\end{equation}
See \cite[Theorem 8.14]{heinonen2001lectures} for the first inequality and \cite[Section 10.6.4]{burago2001course} for the second inequality.

Given $n\in \mathbb{N}$, a metric space $X$ is said to be \emph{$n$-rectifiable} if there exist countably many subsets $E_i\subset \mathbb{R}^n$ and Lipschitz maps $f_i:E_i\to X$ such that $\mathcal{H}^n(X\setminus \cup_i f_i(E_i))=0$; see \cite{bate2017characterizations,kirchheim1994rectifiable,ambrosio2000rectifiable}.

For $\lambda\in (0,1]$ and a base point $x\in X$, we call the rescaled pointed space $(X,\mathsf d/\lambda,x)$ a \emph{blow-up} of $(X,\mathsf d,x)$ at $x$. 
We denote by $\mathrm{Tan}(X,\mathsf d,x)$ the collection of all pointed Gromov--Hausdorff limits of sequences of blow-ups $\{(X,\mathsf d/\lambda_i,x)\}_i$, where $\{\lambda_i\}_i\subset (0,1]$ converges to $0$, and call each element of $\mathrm{Tan}(X,\mathsf d,x)$ a \emph{Gromov--Hausdorff tangent cone} of $X$ at $x$.

\subsection{S-concavity and Busemann concavity}\label{sect:S_and_Busemann_concavity}

\noindent
In this subsection, we briefly recall the definitions of $S$-concavity and Busemann concavity.
Both notions can be viewed as generalizations of non-negative sectional curvature in Finsler geometry.
For more details and examples, we refer to \cite{ohta2021comparison,ohta2009uniform,ohta2007convexities,kell2019sectional} and references therein.

\begin{definition}[S-concave spaces]\label{def:S_concave}
  We say that a complete geodesic space $(X,\mathsf d)$ is \emph{$S$-concave} for some $S\geq 1$ if for any point $p\in X$ and any constant-speed geodesic $\gamma:[0,1]\to X$, it holds that
  \begin{equation}
    \left|p\gamma(t)\right|^2
    \geq
    (1-t)\left|p\gamma(0)\right|^2 + t\left|p\gamma(1)\right|^2 - S t (1-t)\left|\gamma(0)\gamma(1)\right|^2,\quad \text{for any }t\in [0,1].
  \end{equation}
  We call $S$ the \emph{uniform smoothness constant} of $X$.
\end{definition}
Observe that no geodesic space, except a singleton, can be $S$-concave when $S<1$ (see \cite[Ex. 8.12]{ohta2021comparison}).
Typical examples of $S$-concave spaces include $\ell_p^n, \ell_p$, and $L_p$ spaces for $p\in [2,\infty)$ (with $S=p-1$), $\mathrm{CBB}(0)$ spaces (with $S=1$), complete Berwald spaces with non-negative flag curvature whose uniform smoothness constants are bounded above by $S$ (see \cite[Definition 8.14 and Corollary 8.20]{ohta2021comparison}), and their product spaces.

\begin{definition}[Busemann concave spaces]\label{def:Busemann_concave}
  We say that a complete geodesic space $(X,\mathsf d)$ is \emph{Busemann concave} if for any pair of constant-speed geodesics $\gamma,\eta:[0,1]\to X$ starting from a common point, the function
  \begin{equation}\label{eq:Busemann_concave}
    t\mapsto \frac{\left|\gamma(t)\eta(t)\right|}{t}
  \end{equation}
  is non-increasing on $(0,1]$.
\end{definition}
Note that the Busemann concavity \eqref{eq:Busemann_concave}, which can be defined via comparison triangles similarly to Alexandrov spaces, is a strictly weaker condition than the triangle comparison property required for Alexandrov spaces with non-negative curvature; see \cite[Remark 3.5]{han2025structure}.
It follows directly from the definition that Busemann concave spaces are non-branching.
Typical examples of Busemann concave spaces include strictly convex Banach spaces, Alexandrov spaces with non-negative curvature, and product spaces of these spaces.

We remark that, while one can prove that connected complete Berwald spaces with flag curvature bounded below satisfy Busemann concavity locally up to the conjugate radius, it is unknown whether all such spaces with non-negative flag curvature satisfy Busemann concavity globally.
In contrast, all simply connected complete Berwald spaces with absolutely homogeneous Finsler metric and non-positive flag curvature indeed satisfy global Busemann convexity; see \cite[Remark 10]{kozma2004dispersing}.
We refer to \cite{kell2019sectional,kell2015Berwald} and \cite[Remark 3.7]{han2025structure} for further discussion.

\subsection{Notions of angle}\label{subsect:notions_of_angles}

\noindent
In this subsection, we recall two notions of angle introduced in \cite{han2025structure}: \emph{angles viewed from a fixed point} and \emph{angles of fixed scale}, defined on $S$-concave spaces and Busemann concave spaces, respectively.
Although these notions do not coincide in general, they capture complementary geometric information in $S$-concave Busemann concave spaces and are central to the strainer theory developed below. 

We first recall the notion of angle viewed from a fixed point, which is used to measure the orthogonality of geodesics. 
\begin{definition}[Angles viewed from a fixed point]\label{def:angle_viewed_from_fixed_point}
  Let $(X,\mathsf d)$ be an $S$-concave space with $S\geq 1$ and $p\in X$ be a point.
  Let $\gamma:[0,l]\to X$ be a unit-speed geodesic starting from $x\in X$ distinct from $p$.
  Then the angle $\angle px\gamma$, referred to as the \emph{angle viewed from $p$ at $x$ along $\gamma$}, is defined as 
  \begin{equation}\label{eq:angle_viewed_from_fixed_point}
    \angle px\gamma:=\lim_{t\searrow 0}\tilde{\angle}px\gamma(t),
  \end{equation}
  where $\tilde{\angle}px\gamma(t)$ is the angle of the Euclidean comparison triangle $\tilde{\Delta}px\gamma(t)$ at $\tilde{x}$.
\end{definition}

Although the comparison angles $t\mapsto \tilde{\angle}px\gamma(t)$ need not be monotone, it was shown in \cite{han2025structure} that the angle $\angle px\gamma$ is well-defined on $S$-concave spaces.
In the following, we collect several properties of angles viewed from a fixed point that will be used repeatedly in the subsequent sections.
For detailed proofs, we refer to \cite[Section 4.1]{han2025structure}.

\begin{lemma}[\cite{han2025structure}]\label{lem:angle_viewed_from_fixed_point}
  Let $(X,\mathsf{d})$ be an $S$-concave space with $S\geq 1$, and let $p\in X$ be a point and $\gamma$ be a unit-speed geodesic starting from a point $x\in X$ distinct from $p$.
  \begin{enumerate}[fullwidth, label=(\roman*)]
    \item \textbf{Almost comparison inequality.}\label{item:almost_comparison} 
    The angle $\angle px\gamma$ satisfies the following almost comparison inequality:
    \begin{equation}
      \tilde{\angle}px\gamma(t)
      \leq
      \angle px\gamma + \delta_S(t;|px|),\quad \text{for any }t\in (0,t_0],
    \end{equation}
    where $\delta_S(t;|px|):=\arccos(1-\frac{(S-1)t}{2|px|})$ is a non-negative continuous function on $(0,t_0]$, and $t_0$ is a constant depending only on the distance $|px|$, the constant $S$, and the length $l(\gamma)$.

    \item \textbf{First variation formula.}\label{item:first_variation} 
    The function $t\mapsto |p\gamma(t)|$ is differentiable at $t=0$ and satisfies
    \begin{equation}
      \frac{d^+}{dt}\bigg|_{t=0}\left|p\gamma(t)\right|:=\lim_{t\searrow 0}\frac{\left|p\gamma(t)\right| - \left|p\gamma(0)\right|}{t} = - \cos \angle px\gamma.
    \end{equation}

    \item \textbf{Sum of adjacent angles viewed from a common point.}\label{item:sum_adjacent_angles} 
    The sum of adjacent angles viewed from a common fixed point is at most $\pi$.
    More precisely, for any unit-speed geodesic $c$ and any interior point $y\in c$, it holds
    \begin{equation}
      \angle py\xi + \angle py\eta \leq \pi,
    \end{equation}
    where $\xi$ and $\eta$ are re-parametrizations by arc-length of two geodesic segments of $c$ separated by $y$ with the common starting point $\xi(0)=\eta(0)=y$. 

    \item \textbf{Lower semi-continuity.}\label{item:lower_semicontinuity} 
    The angle $\angle px\gamma$ is lower semi-continuous with respect to the point $p$ and the unit-speed geodesic $\gamma$.
    More precisely, given a sequence of points $\{p_i\}_i\subset X$ and a sequence of constant-speed geodesics $\{\gamma_i\}_i\subset X$ such that $\gamma_i$ converges pointwise to a non-trivial constant-speed geodesic $\gamma$ and $p_i$ converges to a point $p\neq \gamma(0)$, it holds
    \begin{equation}
      \angle px\xi \leq \liminf_{i\to \infty}\angle p_ix_i\xi_i,
    \end{equation} 
    where $\xi_i$ and $\xi$ denote the arc-length re-parametrizations of $\gamma_i$ and $\gamma$, respectively, such that $\xi_i(0)=\gamma_i(0)=x_i$ and $\xi(0)=\gamma(0)=x$.
  \end{enumerate}
\end{lemma}
\begin{remark}[Asymmetry of angles viewed from a point]\label{rmk:asymmetry_angles_viewed_from_fixed_point}
  In general, angles viewed from a fixed point are not \emph{symmetric}.
  More precisely, given three distinct points $p,x,q\in X$ and geodesics $\xi,\eta$ connecting $x$ to $p$ and $q$, respectively, $\angle px\eta$ and $\angle qx\xi$ need not coincide, even when $X$ is uniquely geodesic.
  For example, let $a:=(1,2), b:=(4,-1)$ be two points in $\ell^2_3:=(\mathbb{R}^2, \|\cdot\|_3)$, and let $\xi$ and $\eta$ be the unit-speed geodesic segments from the origin $o$ to $a$ and $b$, respectively.
  A direct computation gives
  \begin{equation}
      \angle ao\eta=\pi/2,
      \qquad
      \angle b o\xi=\arccos\left(\frac{14}{195^{2/3}}\right)
        \in\left(0,\frac{\pi}{2}\right).
  \end{equation}
  This example also illustrates the non-symmetry of \emph{Birkhoff--James orthogonality}, introduced by Birkhoff \cite{birkhoff1935orthogonality} and further developed by James \cite{james1945orthogonality,james1947orthogonality} for general normed spaces.
  We refer to \cite{grover2021birkhoff} for a recent survey on Birkhoff--James orthogonality and its relationship with inner product structures on normed spaces.
  See \cite{kell2016symmetric} for an extension of Birkhoff--James orthogonality to Busemann convex spaces.
\end{remark}


We now turn to another notion of angle defined on Busemann concave spaces, called the \emph{angle of fixed scales}. 
This notion of angle is connected to the geometry of tangent cones of Busemann concave spaces; see \cite[Section 4.2]{han2025structure} for more details.
\begin{definition}[Angles of fixed scale]\label{def:angles_of_fixed_scales}
  Let $(X,\mathsf d)$ be a Busemann concave space, $x\in X$ be a point, and let $\xi,\eta:[0,1]\to X$ be two non-trivial unit-speed geodesics starting from $x$.
  For any $t,s>0$, the angle $\angle_x (\xi(t),\eta(s))$, referred to as the \emph{angle between $\xi$ and $\eta$ at $x$ of scales $(t,s)$}, is defined as
  \begin{equation}\label{eq:angles_of_fixed_scales}
    \angle_x\left(\xi(t), \eta(s)\right):=\sup_{\substack{\theta\in (0,1]\\\max\{\theta t,\theta s\}\leq a}} \tilde{\angle}_x \left(\xi(\theta t), \eta(\theta s)\right),
  \end{equation}
  where $a>0$ is an arbitrary positive number such that both $\xi$ and $\eta$ are defined on the interval $I_a:=[0,a]$.
  We call $\angle_x(\xi(t),\eta(s))$ the \emph{angle of common scale} if the scales $t$ and $s$ coincide.
\end{definition}
Note that, in the definition of angles of fixed scale, the scales $t$ and $s$ are allowed to lie outside the domains of the unit-speed geodesics $\xi$ and $\eta$.
By Busemann concavity, angles of fixed scale are well-defined, and the supremum in \eqref{eq:angles_of_fixed_scales} is in fact a limit and is independent of the choice of $a$.
Furthermore, angles of fixed scale are positively scaling-invariant, i.e., $\angle_x(\xi(\lambda t),\eta(\lambda s))=\angle_x(\xi(t),\eta(s))$ for any $\lambda>0$ and any non-trivial unit-speed geodesics $\xi,\eta$ starting from $x$.
Therefore, it depends only on the ratio $t/s$, rather than on the individual values of $t$ and $s$.

Finally, we point out that on $S$-concave Busemann concave spaces, these two notions of angle are generally different, even when the underlying spaces are uniquely geodesic; see \cite[Example 4.14]{han2025structure}.

\subsection{Tangent cones and spaces of directions with common length}\label{subsect:tangent_cones}

\noindent
In this subsection, we recall the notions of tangent cones and spaces of directions with common length on Busemann concave spaces; see \cite[Section 2.3]{kell2019sectional} and \cite[Section 4.2]{han2025structure} for further details.
\begin{definition}[Tangent cones on Busemann concave spaces]\label{def:tangent_cones_Busemann_concave}
  Let $(X,\mathsf{d})$ be a Busemann concave space.
  We denote by $\Gamma_x$ the collection of all non-trivial maximal unit-speed geodesics starting from $x\in X$.
  The pre-tangent cone $\hat{T}_xX$ at $x$ is defined as the set $\Gamma_x\times [0,\infty)/\sim$, where all points of the form $(\gamma, 0), \gamma\in \Gamma_x$, are identified with the common point $o$.
  The metric $\mathsf{d}_x$ on $\hat{T}_xX$ is defined as follows: for any two points $v:=(\gamma,t), w:=(\eta,s)\in \hat{T}_xX$, let $I_a:=[0,a]$ be an interval such that both $\gamma$ and $\eta$ are defined on $I_a$.
  The $\mathsf{d}_x$-distance between $v$ and $w$ is defined as
  \begin{equation}\label{eq:tangent_cones_Busemann_concave}
    \mathsf d_x\left((\gamma,t), (\eta,s)\right):=\sup_{\substack{\theta\in (0,1],\\\max\{\theta t, \theta s\}\leq a}}\frac{|\gamma(\theta t),\eta(\theta s)|}{\theta}.
  \end{equation}
  The tangent cone $T_xX$ at $x$ is defined as the completion of $\hat{T}_xX$ with respect to the metric $\mathsf d_x$.
  For $v\in T_xX$, we denote by $|v|_x:=\mathsf d_x(v,o)$ the distance from $v$ to $o$.
  We call the point $o$ the \emph{apex} of the tangent cone $T_xX$.
\end{definition}
As shown in \cite[Lemma 2.17]{kell2019sectional}, the metric $\mathsf{d}_x$ is well-defined on tangent cones of Busemann concave spaces, and the supremum in \eqref{eq:tangent_cones_Busemann_concave} is in fact a limit, which is independent of the choice of $a$.
Moreover, the metric $\mathsf{d}_x$ is positively homogeneous and is related to the angle of fixed scale through the Euclidean law of cosines:
\begin{equation}
  \mathsf{d}_x\left((\gamma,t), (\eta, s)\right)^2
  =
  t^2 + s^2 - 2 t s \cos \angle_x (\gamma(t), \eta(s)),\quad \text{for any } (\gamma,t),(\eta,s)\in \hat{T}_xX,
\end{equation}
see \cite[Lemma 4.16]{han2025structure}.
Note that, as observed by Kell in \cite[Section 2.3]{kell2019sectional}, the tangent cones defined in Definition~\ref{def:tangent_cones_Busemann_concave} do not necessarily coincide with the Gromov--Hausdorff tangent cones of Busemann concave spaces, even when the underlying spaces are compact.
However, when $X$ is (locally) doubling, these two notions of tangent cone do coincide.
In particular, each point admits a unique tangent cone in this case; see \cite[Corollary 2.21]{kell2019sectional} (see also \cite[Proposition 4.7]{han2025structure}).

We now turn to the geometry of tangent cones in Busemann concave spaces.
The dependence of angles of fixed scale on the scale ratio prevents us from identifying the space of directions at $x$ with the set of equivalence classes of unit-speed geodesics starting from $x$ equipped with the angle of fixed scale as a metric, as in the case of Alexandrov spaces.
Consequently, tangent cones of Busemann concave spaces have a more intricate structure than the metric cone over the space of directions.
Motivated by the study of the geometry of Banach spaces through unit spheres, we introduced in \cite{han2025structure} the following notion of \emph{spaces of directions with common length} to better understand the geometry of tangent cones of Busemann concave spaces.
\begin{definition}[Spaces of directions with common length]\label{def:space_of_directions_common_length}
  Let $(X,\mathsf d)$ be a Busemann concave space and let $x\in X$ be a point.
  Given $l>0$, we denote by $\hat{\Sigma}_x^lX$ the subset of the pre-tangent cone $\hat{T}_xX$ consisting of all elements of the form $(\gamma,l)\in \hat{T}_xX$.
  The \emph{space of directions at $x$ with common length $l$}, denoted by $\Sigma_x^lX$, is defined as the closed subset of the tangent cone $T_xX$ consisting of all elements $v\in T_xX$ whose distance from $o$ is $l$.
  We define $\angle_x(\cdot,\cdot): \Sigma_x^lX \times \Sigma_x^lX \to [0,\pi]$ by
  \begin{equation}
    \angle_x(v,w):=\arccos\left(\frac{|v|_x^2 + |w|_x^2 - \mathsf{d}_x(v,w)^2}{2|v|_x |w|_x}\right),\quad \text{for any }v,w\in \Sigma_x^lX.
  \end{equation}
\end{definition}

Recall that a function $\rho:E\times E\to \mathbb{R}^+$ on a set $E$ is called a \emph{quasi-metric}\footnote{Some authors also refer to such generalized metric functions as a \emph{semi-metric}; see, for example, \cite{willard70topology}.} if $\rho$ satisfies all the axioms of a metric except that the triangle inequality is replaced by a relaxed one, i.e., $\rho(x,y)\leq C(\rho(x,z) + \rho(y,z))$ holds for any $x,y,z\in E$ for some $C\geq 1$; see, for example, \cite{xia2009geodesic}.
It is shown in \cite[Lemma 4.19]{han2025structure} that $(\Sigma_x^lX,\angle_x)$ is a quasi-metric space when equipped with $\angle_x(\cdot,\cdot)$. 
More precisely, $\angle_x(\cdot,\cdot)$ satisfies the following relaxed triangle inequality:
\begin{equation}
  \angle_x(u,v)\leq 2\left(\angle_x(u,w) + \angle_x(w,v)\right),\quad \text{for any }u,v,w\in \Sigma_x^lX.
\end{equation}
Furthermore, the subset $\hat{\Sigma}_x^lX$ is dense in $\Sigma_x^lX$ with respect to the angle quasi-metric $\angle_x(\cdot,\cdot)$, and $\angle_x(\cdot,\cdot)$ is positively scaling-invariant; that is,
\begin{equation}
  \angle_x((\gamma,\lambda l),(\eta,\lambda l))=\angle_x((\gamma,l),(\eta,l)),
\end{equation}
for any $\lambda>0$ and $(\gamma,l),(\eta,l)\in \hat{\Sigma}_x^lX$.

The quasi-metric structure of spaces of directions with common length enables us to prove the uniform compactness of these spaces in doubling Busemann concave spaces, which is used to prove the Hausdorff dimension estimates for singular sets.
We refer to \cite[Lemma 4.20]{han2025structure} for a detailed proof.
\begin{lemma}\label{lem:uniform_compactness_space_of_directions_common_length}
  Let $(X,\mathsf d)$ be a doubling Busemann concave space and $x\in X$.
  Then the family of quasi-metric spaces $\{(\Sigma_x^lX, \angle_x)\}_{l>0}$ is uniformly compact in the sense that for any $\varepsilon>0$, there exists a constant $N_0(\varepsilon)>0$, depending only on $\varepsilon$ and the doubling constant of $X$, such that every $\varepsilon$-separated subset of $(\Sigma_x^lX, \angle_x)$ has cardinality at most $N_0(\varepsilon)$ for any $l>0$.
\end{lemma}

\section{Weak quadruple condition}\label{sect:weak_quadruple_condition}
\noindent
This section forms the comparison-theoretic core of the paper. 
After introducing the \emph{$(\varepsilon,\delta)$-weak quadruple condition}, we prove its compatibility with $S$-concavity: for every $\varepsilon>0$, every $S$-concave space with $S\geq 1$ satisfies the $(\varepsilon,\delta)$-weak quadruple condition, provided that $\delta>0$ is sufficiently small.

\medskip

Recall that $\delta_S(s;t):=\arccos(1-(S-1)s/(2t))$ is the non-negative error function defined in Lemma~\ref{lem:angle_viewed_from_fixed_point}.
\begin{definition}[Weak quadruple condition]\label{def:weak_quadruple_condition}
  Let $X$ be an $S$-concave space.
  Given $\varepsilon\in [0,\pi],\ \delta\in (0,\pi]$, we say that $X$ satisfies the {\em $(\varepsilon,\delta)$-weak quadruple condition} if the following holds: for any four distinct points $(x;p, y, z)$ in $X$, satisfying $\tilde{\angle}pxy \geq \pi-\delta$, $|xz|\leq |xy|\leq |px|$, and
  \begin{equation}\label{eq:weak_quadruple_condition_1}
    \delta_S(|xz|;|xy|)\leq \delta,\quad \delta_S(|xy|;|px|)\leq \delta,
  \end{equation}
  it holds
  \begin{equation}
    \tilde{\angle}pxy + \tilde{\angle}pxz + \tilde{\angle}yxz \leq 2\pi + \varepsilon.
  \end{equation}
  We call $\delta$ the \emph{straightness parameter} and $\varepsilon$ the \emph{angle-excess parameter}.
\end{definition}

It is immediate from the definition that if an $S$-concave space satisfies the $(\varepsilon,\delta)$-weak quadruple condition, then it also satisfies the $(\varepsilon',\delta')$-weak quadruple condition for every $\delta'<\delta$ and $\varepsilon'>\varepsilon$.
Moreover, any $S$-concave space with $S=1$, that is, any Alexandrov space with non-negative curvature, satisfies the strongest form, namely the $(0,\pi)$-weak quadruple condition.
\begin{remark}
  Our definition is weaker than the relaxed quadruple condition appearing in \cite[Remark 6.9]{burago1992ad}: that condition relaxes the upper bound for the sum of comparison angles to $2\pi$ plus a dimension-dependent error term, but imposes no restriction on the configuration of points as well as on the near-collinearity of quadruples.
\end{remark}

\begin{remark}
  We deliberately impose restrictions on the configuration of points, together with the near-collinearity condition, so that the $(\varepsilon,\delta)$-weak quadruple condition can hold for nontrivial $\varepsilon>0$ in $S$-concave spaces with large $S$.
  For example, consider the points $o=(0,0)$, $a=(0,1)$, $b=(1,-1)$, and $c=(-1,-1)$ in the space $\ell^2_p:=(\mathbb{R}^2,\|\cdot\|_p)$ for $p\in (2,\infty)$.
  A direct computation shows that $\tilde{\angle}aob=\tilde{\angle}aoc=\arccos\!\left(\frac{1+2^{1/p}-(1+2^p)^{2/p}}{2^{1+1/p}}\right)\nearrow \pi$ and $\tilde{\angle}boc=\arccos(1-2^{1-2/p})\nearrow \pi$  as $p\to \infty$.
  In particular, $\tilde{\angle}aob + \tilde{\angle}aoc + \tilde{\angle}boc \to 3\pi$ as $p\to \infty$.
\end{remark}

In the remainder of this section, we show that every $S$-concave space automatically satisfies the $(\varepsilon,\delta)$-weak quadruple condition for any $\varepsilon>0$, provided that $\delta>0$ is chosen sufficiently small.
Before stating the precise result, we introduce the following notion of weak upper angle, which is a variant of the usual upper angle; see, for example, \cite[Definition I.1.12]{bridson1999metric}.


\begin{definition}[$S$-upper angles]\label{def:S_upper_angles}
  Let $X$ be an $S$-concave space with $S\geq 1$.
  Given two unit-speed geodesics $\eta,\xi$ emanating from a common point $x\in X$, we denote by $\angle_S(\eta,\xi)$ the $S$-upper angle from $\eta$ to $\xi$ at $x$, defined as
  \begin{equation}
    \angle_S(\eta,\xi):=\sup_{s>0}\angle \eta(s)x\xi.
  \end{equation}
\end{definition}
The supremum in the definition is actually realized as a limit, as implied by the following lemma.
\begin{lemma}\label{lemma:monotonicity_angle_view_point}
  Let $X$ be an $S$-concave space with $S\geq 1$, and let $\eta,\xi\subset X$ be two unit-speed geodesics emanating from a common point $x\in X$.
  Then the function $s\mapsto \angle \eta(s)x\xi$ is non-increasing on $(0,l(\eta))$.
  In particular, the $S$-upper angle $\angle_S(\eta,\xi)$ is equal to $\lim_{s\searrow 0}\angle \eta(s)x\xi$.
\end{lemma}
\begin{proof}
  Let $0<s<t\leq l(\eta)$.
  By the triangle inequality of distance and the first variation formula (Lemma~\ref{lem:angle_viewed_from_fixed_point}~\ref{item:first_variation}), it follows that
  \begin{multline}
    \cos \angle \eta(t)x\xi
    =
    \lim_{\tau \searrow 0}\frac{|\eta(t)x|-|\eta(t)\xi(\tau)|}{\tau}\\
    \geq
    \lim_{\tau \searrow 0}\frac{|\eta(t)\eta(s)|+|\eta(s)x|-|\eta(t)\eta(s)|-|\eta(s)\xi(\tau)|}{\tau}\\
    =
    \lim_{\tau \searrow 0}\frac{|\eta(s)x|-|\eta(s)\xi(\tau)|}{\tau}
    =
    \cos \angle \eta(s)x\xi.
  \end{multline}
  This implies that $\angle \eta(t)x\xi \leq \angle \eta(s)x\xi$, and the claim follows.
\end{proof}
\begin{remark}
  In general, a $S$-upper angle is not symmetric: $\angle_S(\eta,\xi)$ need not coincide with $\angle_S(\xi,\eta)$.
  Moreover, unlike usual upper angles (see \cite[Proposition I.1.14]{bridson1999metric}), it need not satisfy the triangle inequality.
  A $S$-upper angle also need not coincide with the usual upper angle.
\end{remark}


We are now ready to present the main result of this subsection.
We establish the compatibility through a quantitative estimate on the sum of comparison angles of quadruples in terms of the straightness parameter $\delta$ and the uniform smoothness constant $S$.
The main difficulty is that neither the $S$-upper angles nor angles viewed from a point satisfy the triangle inequality in general.
To overcome this obstacle, we establish a relaxed triangle inequality that bounds the angle viewed from a point in terms of the corresponding $S$-upper angles, up to an error controlled by $\delta$ and $S$, provided that the angle viewed from a point is sufficiently close to $\pi$.
This estimate is a key ingredient in the proof.
Following a similar strategy to \cite[Proposition I.1.14]{bridson1999metric}, we construct an auxiliary comparison triangle in $\mathbb{R}^2$ and derive a contradiction if the relaxed triangle inequality fails.
\begin{proposition}[Compatibility]\label{lemma:verification_WQC_II}
  Let $X$ be an $S$-concave space with $S\geq 1$, and let $0<\delta\leq 1/64$ be a small number such that $4(S-1)\sqrt{\delta}\leq 1$.
  Let $(x;p,y,z)$ be four distinct points in $X$ satisfying the conditions of Definition~\ref{def:weak_quadruple_condition} with the straightness parameter $\delta$.
  Then it holds
  \begin{equation}\label{eq:verification_WQC_II}
    \tilde{\angle}pxy + \tilde{\angle}pxz + \tilde{\angle}yxz
    \leq
    2\pi + \arccos\left(1-4(S-1)\sqrt{\delta}\right) + 14\delta + 2\sqrt{\delta}.
  \end{equation}
  In particular, given any $\varepsilon\in (0,\pi]$, there exists $\bar{\delta}:=\bar{\delta}(\varepsilon,S)>0$ such that $X$ satisfies the $(\varepsilon,\delta)$-weak quadruple condition for all $\delta\in (0,\bar{\delta}]$.
\end{proposition}
\begin{proof}
  We first establish the inequality \eqref{eq:verification_WQC_II}.
  Let $\eta_{xz}$ be a unit-speed geodesic from $x$ to $z$.
  Take $x'\in \eta_{xz}$ sufficiently close to $x$ such that $\tilde{\angle}px'y>\pi-2\delta$, $|xx'|<|x'z|$, and
  \begin{equation}\label{eq:Verification_WQC_II_App_000}
    \delta_S(|x'y|;|px'|)<2\delta,\quad \delta_S(|x'z|;|x'y|)<2\delta,\quad \delta_S(|x'z|;|px'|)<2\delta.
  \end{equation}
  Let $\gamma$ be a unit-speed geodesic from $x'$ to $y$, and let $\xi,\eta$ be the unit-speed geodesic segments of $\eta_{xz}$ from $x'$ to $x$ and $z$, respectively.

  \begin{enumerate}[fullwidth, label=\textbf{Step \arabic*}.]
    \item We prove the following relaxed triangle inequality by contradiction:
    \begin{equation}\label{eq:Verification_WQC_II_App_0}
      \angle px'\gamma \leq \angle px'\xi + \sqrt{\delta} + \angle_S(\gamma,\xi) + \arccos(1-4(S-1)\sqrt{\delta})+\sqrt{\delta} + 8\delta.
    \end{equation}
    Suppose the inequality \eqref{eq:Verification_WQC_II_App_0} does not hold.
    By Lemma~\ref{lem:angle_viewed_from_fixed_point}~\ref{item:almost_comparison} and \eqref{eq:Verification_WQC_II_App_000}, we have
    \[
      \angle px'\gamma\geq \tilde{\angle}px'y - 2\delta > \pi -4\delta.
    \]
    Thus, we can find some $\alpha\in (\pi-4\delta, \pi)$ such that
    \begin{equation}\label{eq:Verification_WQC_II_App_00}
      \angle px'\gamma > \alpha > \angle px'\xi + \sqrt{\delta} + \angle_S(\gamma,\xi) + \arccos(1-4(S-1)\sqrt{\delta})+ \sqrt{\delta} + 8\delta.
    \end{equation}
    Let
    \begin{equation}
      \alpha_1:=\angle px'\xi + \sqrt{\delta}\quad\text{and}\quad  \alpha_2:=\angle_S(\gamma,\xi) + \arccos(1-4(S-1)\sqrt{\delta})+\sqrt{\delta}.
    \end{equation}
    By the definition of the angle viewed from a fixed point $\angle px'\gamma=\lim_{t\to 0}\tilde{\angle}px'\gamma(t)$, we can choose $y'\in \gamma$ close to $x'$ such that
    \begin{equation}\label{eq:Verification_WQC_II_App_001}
      \tilde{\angle}px'y'>\alpha,\quad |x'y'|<|x'x|,\quad \text{and}\quad  |x'y'|/|px'|<\delta/2.
    \end{equation}
    
    In the following, we construct an auxiliary comparison triangle in $\mathbb{R}^2$ to derive a contradiction.
    Let $a,b\in \mathbb{R}^2$ be two points such that $|oa|_2=|px'|, |ob|_2=|x'y'|$ and $\tilde{\angle}aob=\alpha$, where $|\cdot|_2$ denotes the distance induced by the standard Euclidean norm on $\mathbb{R}^2$.
    Since $\alpha\in (\pi-4\delta, \pi)$, the auxiliary comparison triangle $\tilde{\Delta}aob$ is non-degenerate.
    Furthermore, by the third inequality in \eqref{eq:Verification_WQC_II_App_001} and our choice of $\alpha$, we have
    \begin{equation}\label{eq:Verification_WQC_II_App_002}
      \tilde{\angle}oab<\delta\quad \text{and}\quad \tilde{\angle}oba<4\delta.
    \end{equation}
    The hypothesis \eqref{eq:Verification_WQC_II_App_00} gives $\alpha_1+\alpha_2<\pi-8\delta$.
    Therefore, at most one of $\alpha_1$ and $\alpha_2$ can exceed $\pi/2-4\delta$.
    We now distinguish several cases.
    \begin{enumerate}[fullwidth, label=\textbf{\textit{Case \arabic*}}.]
      \item $\alpha_2\geq \pi/2 - 4\delta$.
      Then it follows that $\alpha_1<\pi/2-4\delta$, and therefore by \eqref{eq:Verification_WQC_II_App_002}, we have $\alpha_1 + \tilde{\angle}oab < \pi/2-4\delta + \delta<\pi/2-3\delta$.
      Choose $c\in \mathbb{R}^2$ to be the point on the geodesic segment $[ab]$ such that $\tilde{\angle}aoc=\alpha_1$.
      By elementary plane trigonometry and the Euclidean law of sines to the triangle $\tilde{\Delta}obc$, it follows that
      \begin{equation}
        \frac{|oc|_2}{\sin \tilde{\angle}obc}=\frac{|ob|_2}{\sin \tilde{\angle}ocb}=\frac{|ob|_2}{\sin(\alpha_1 + \tilde{\angle}oab)}.
      \end{equation}
      Note that $\alpha_1=\angle px'\xi + \sqrt{\delta}\geq \sqrt{\delta}$.
      Thus, by the inequality $1/2\leq \sin x/x\leq 1$ and the monotonicity of the sine function on $[0,\pi/2]$, it follows that
      \begin{equation}
        \frac{|oc|_2}{|ob|_2}=\frac{\sin\tilde{\angle}obc}{\sin(\alpha_1+\tilde{\angle}oab)}
        \leq
        \frac{\sin \tilde{\angle}oba}{\sin \alpha_1}
        \leq
        \frac{4\delta}{\sin\sqrt{\delta}}
        \leq
        8\sqrt{\delta}.
      \end{equation}
      The inequality above, together with the assumption on the value of $\delta$, implies that $|oc|_2\leq|ob|_2=|x'y'|\leq|x'x|$.
      Therefore, we can find a point $z'\in \xi$ such that $|x'z'|=|oc|_2$.
      In particular, we have $|x'z'|/|x'y'|\leq 8\sqrt{\delta}$, which further implies that
      \begin{equation}
        \delta_S\left(|x'z'|;|x'y'|\right)
        =
        \arccos\left(1-\frac{(S-1)|x'z'|}{2|x'y'|}\right)
        \leq
        \arccos\left(1-4(S-1)\sqrt{\delta}\right).
      \end{equation}
      By the almost comparison inequality (Lemma~\ref{lem:angle_viewed_from_fixed_point} \ref{item:almost_comparison}), the monotonicity of $S$-upper angles (Lemma~\ref{lemma:monotonicity_angle_view_point}) and the inequality \eqref{eq:Verification_WQC_II_App_000}, we have
      \begin{equation}\label{eq:Verification_WQC_II_App_1}
        \tilde{\angle}aoc=\alpha_1=\angle px'\xi + \sqrt{\delta}
        \geq
        \angle px'\xi + 2\delta
        \geq
        \angle px'\xi + \delta_S(|x'z'|;|px'|)
        \geq
        \tilde{\angle}px'z',
      \end{equation}
      and
      \begin{multline}\label{eq:Verification_WQC_II_App_2}
        \tilde{\angle}boc
        =\alpha -\alpha_1>\alpha_2
        =
        \angle_S(\gamma,\xi) + \arccos(1-4(S-1)\sqrt{\delta}) + \sqrt{\delta}\\
        \geq
        \angle y'x'\xi + \delta_S(|x'z'|;|x'y'|)
        \geq
        \tilde{\angle} y'x'z'.
      \end{multline}
      These inequalities imply that $|ac|_2\geq |pz'|$ and $|bc|_2\geq |y'z'|$.
      Therefore, by our construction of the auxiliary triangle $\tilde{\Delta}aob$ and the hypothesis that $\tilde{\angle}px'y'>\alpha$, we have
      \begin{equation}
        |py'|
        >
        |ab|_2
        =
        |ac|_2 + |bc|_2
        \geq
        |pz'| + |z'y'|,
      \end{equation}
      which contradicts the triangle inequality of distance.

      \item $\alpha_2<\pi/2 -4\delta$.
      In this case, we have $\alpha_2+ \tilde{\angle}oba<\pi/2$.
      Choose $c\in \mathbb{R}^2$ to be the point on the geodesic segment $[ab]$ such that $\tilde{\angle}boc=\alpha_2$.
      By applying the Euclidean law of sines to the plane triangle $\tilde{\Delta}obc$, we obtain
      \begin{equation}
        \frac{|oc|_2}{\sin \tilde{\angle}obc} = \frac{|ob|_2}{\sin\tilde{\angle}ocb}
        =
        \frac{|ob|_2}{\sin\left(\pi-\alpha_2 - \tilde{\angle}oba\right)}.
      \end{equation}
      Note that $\alpha_2 \geq \sqrt{\delta}$.
      Therefore, by the inequality $1/2\leq \sin x/x\leq 1$ and the monotonicity of the sine function on $[0,\pi/2]$, we have
      \begin{equation}
        \frac{|oc|_2}{|ob|_2}
        =
        \frac{\sin\tilde{\angle}obc}{\sin\left(\pi-\alpha_2 - \tilde{\angle}oba\right)}
        =
        \frac{\sin\tilde{\angle}oba}{\sin\left(\alpha_2 + \tilde{\angle}oba\right)}
        \leq
        \frac{\sin(4\delta)}{\sin \sqrt{\delta}}
        \leq
        8\sqrt{\delta}.
      \end{equation}
      Thus, by the same argument as Case 1, we can find a point $z'\in \xi$ such that $|x'z'|=|oc|_2$ and $|x'z'|/|x'y'|\leq 8\sqrt{\delta}$.
      By the same arguments as in \eqref{eq:Verification_WQC_II_App_1} and \eqref{eq:Verification_WQC_II_App_2}, we have
      \begin{equation}
        \tilde{\angle}aoc
        =\alpha-\alpha_2
        >\alpha_1
        \geq
        \tilde{\angle}px'z',\quad \text{and}\quad
        \tilde{\angle}boc
        =\alpha_2
        \geq
        \tilde{\angle}y'x'z',
      \end{equation}
      which further implies that $|py'|>|pz'|+|z'y'|$.
      This contradicts the triangle inequality of distance.
    \end{enumerate}
    
    In both cases, we arrive at a contradiction. 
    Therefore, the relaxed triangle inequality \eqref{eq:Verification_WQC_II_App_0} must hold.

    \item We now establish the estimate \eqref{eq:verification_WQC_II}.
    By the almost comparison inequality and the relaxed triangle inequality \eqref{eq:Verification_WQC_II_App_0}, it follows that
    \begin{multline}\label{eq:Verification_WQC_II_App_3}
      \tilde{\angle}px'y + \tilde{\angle}px'z + \tilde{\angle}yx'z
      \leq
      \angle px'\gamma + 2\delta + \angle px'\eta + 2\delta + \angle yx'\eta + 2\delta\\
      \leq
      \angle px'\xi + \angle px'\eta + \angle yx'\eta + \angle_S(\gamma,\xi)\\ 
      + \arccos(1-4(S-1)\sqrt{\delta}) + 2\sqrt{\delta} + 14\delta.
    \end{multline}
    Note that by the sum of adjacent angles (Lemma~\ref{lem:angle_viewed_from_fixed_point}~\ref{item:sum_adjacent_angles}), we have $\angle px'\xi + \angle px'\eta \leq \pi$.
    On the other hand, by the monotonicity of $S$-upper angles (Lemma~\ref{lemma:monotonicity_angle_view_point}) and the sum of adjacent angles, it follows that
    \begin{equation}
      \angle yx'\eta + \angle_S(\gamma,\xi)
      \leq
      \angle_S(\gamma,\eta) + \angle_S(\gamma,\xi)
      =
      \lim_{s\searrow 0}\left(\angle\gamma(s)x'\eta + \angle \gamma(s)x'\xi\right)
      \leq
      \pi.
    \end{equation}
    Plugging these two inequalities into \eqref{eq:Verification_WQC_II_App_3}, we obtain
    \begin{equation}
      \tilde{\angle}px'y + \tilde{\angle}px'z + \tilde{\angle}yx'z
      \leq
      2\pi + \arccos\left(1-4(S-1)\sqrt{\delta}\right) + 2\sqrt{\delta} + 14\delta.
    \end{equation}
    Note that the inequality above holds for all $x'\in \eta_{xz}$ close to $x$.
    Thus, by the continuity of comparison angles, we obtain the desired estimate \eqref{eq:verification_WQC_II}.

    Finally, for the last claim, given any $0<\varepsilon\leq \pi$, let $\bar{\delta}:=\bar{\delta}(\varepsilon,S)$ be the largest positive number such that $0<\bar{\delta}\leq 1/64$, $4(S-1)\sqrt{\bar{\delta}}\leq 1$, and $\arccos(1-4(S-1)\sqrt{\bar{\delta}}) + 2\sqrt{\bar{\delta}} + 14\bar{\delta}\leq \varepsilon$.
    From \eqref{eq:verification_WQC_II} and the monotonicity of the inverse cosine function, our claim follows.
  \end{enumerate}
\end{proof}
\begin{remark}[Modified weak quadruple condition]
  In fact, $S$-concave spaces with $S\geq 1$ satisfy a more flexible variant of the $(\varepsilon,\delta)$-weak quadruple condition.
  All assumptions in Definition~\ref{def:weak_quadruple_condition} remain unchanged, except that the relative distance condition \eqref{eq:weak_quadruple_condition_1} is replaced by
  \begin{equation}
    \max\{|xz|/|xy|, |xy|/|px|\}\leq \delta,
  \end{equation}
  which does not involve the error function $\delta_S$.
  This modified weak quadruple condition can be formulated for general metric spaces beyond the class of $S$-concave spaces and is stable under Gromov--Hausdorff convergence within the class of geodesic spaces.
  In the setting of $S$-concave Busemann concave spaces, however, it is more convenient to use the weak quadruple condition formulated in terms of $\delta_S$. 
  We therefore do not pursue this modified version further in the present paper.
\end{remark}
\begin{remark}
  It is unclear to us whether, for $S$-concave Busemann concave spaces without additional geometric assumptions, the sum of the comparison angles of quadruples with other configurations of points, such as when the near-collinearity condition is imposed on the two points closest to the base point, can still be bounded above by $2\pi$ up to a controlled error term.
\end{remark}

\medskip

\section{Strainers and strainer maps}\label{sect:strainers}
\noindent
In this section, we introduce variants of strainers and strainer maps for $S$-concave Busemann concave spaces.
We then establish several basic properties of these strainer maps, including $\varepsilon$-openness and the bi-Lipschitz property.
Finally, we prove a self-improvement property of strainers, which is needed for the study of the dimension of $S$-concave Busemann concave spaces.

\medskip

The notation in this section involves several small parameters.
For ease of reference, we summarize their roles in Table~\ref{tab:parameter_bookkeeping}.
The precise assumptions are restated in the corresponding propositions and lemmas.
\begin{table}[htbp]
  \centering
  \small
  \begin{tabularx}{\textwidth}{|>{\raggedright\arraybackslash}p{0.17\textwidth}|>{\raggedright\arraybackslash}p{0.27\textwidth}|>{\raggedright\arraybackslash}X|}
    \hline
    Symbol & Role & Where it is used \\
    \hline
    $S$ 
    & Uniform smoothness constant in the $S$-concavity inequality 
    & Controls the error function $\delta_S$ and the admissible weak quadruple parameters obtained in Proposition \ref{lemma:verification_WQC_II}. \\
    \hline
    $\tilde{\varepsilon}$ and $\tilde{\delta}$ 
    & Angle-excess and straightness parameters in the weak quadruple condition 
    & Enter the estimates for almost orthogonality, openness, bi-Lipschitz, and self-improvement properties. \\
    \hline
    $\delta$ and $\alpha$ 
    & Straightness and almost orthogonality parameters of an initial strainer 
    & Specify the initial $(k,\delta,\alpha)$-strainer before the improvement procedure. \\
    \hline
    $\delta_k$ 
    & Dimension-dependent smallness threshold 
    & Chosen in Proposition~\ref{prop:epsilon_openness_strainer_maps} to ensure that the induction establishing the openness of $k$-strainer maps closes. \\
    \hline
    $\varepsilon_k$ and $\bar{\varepsilon}_k$ 
    & Openness (or co-Lipschitz) constants of strainer maps 
    & Determined in Proposition~\ref{prop:epsilon_openness_strainer_maps} and Corollary~\ref{cor:epsilon_openness_strainer_maps}.\\
    \hline
    $\delta'$ and $\varepsilon'$ 
    & Improved straightness and angle-excess parameters 
    & Used in Lemma~\ref{lemma:almost_self_improvement} in the strengthened dequeuing-and-enqueuing procedure. \\
    \hline
  \end{tabularx}
  \caption{
  Parameters used in the strainer theory and their roles.}
  \label{tab:parameter_bookkeeping}
\end{table}

\subsection{Definitions}\label{subsect:strainer_definitions}
\noindent
\begin{definition}[$(1,\delta)$-strainer]\label{def:1_strainer}
  Let $X$ be an $S$-concave space with $S\geq 1$, and $x\in X$.
  Given $0<\delta<1/2$, we say a point $p\in X\setminus \{x\}$ is a \emph{$(1,\delta)$-strainer at $x$} if there exists a point $q\in X\setminus \{x\}$ such that $|qx|< |px|, \delta_S(|qx|;|px|)<\delta$, and the Euclidean comparison angle $\tilde{\angle}pxq> \pi -\delta$.
  In this case, we refer to the point $q$ as an \emph{opposite strainer of $p$ at $x$}, and refer to the pair $(p,q)$ as a \emph{$(1,\delta)$-strainer pair}.
  We say that $x$ is a $(1,\delta)$-strained point if it admits a $(1,\delta)$-strainer at itself.
\end{definition}

In the following, we introduce a slightly finer notion of $(k,\delta)$-strainer, called \emph{$(k,\delta,\alpha)$-strainers}, in which the parameter $\delta$ controls the straightness of strainer pairs, while the parameter $\alpha$ controls the orthogonality between different strainer pairs.
\begin{definition}[$(k,\delta,\alpha)$-strainers]\label{def:k_strainer}
  Let $X$ be an $S$-concave space with $S\geq 1$, and let $0<\delta,\alpha<1/2$.
  Given $k\in \mathbb{N}, k\geq 2$, we call a $k$-tuple $(p_1,\ldots, p_k)$ of points in $X$ a \emph{$(k,\delta,\alpha)$-strainer at $x\in X$} if the following inductive conditions hold:
  \begin{enumerate}[label=(\arabic*)]
    \item The tuple $(p_1,\ldots,p_{k-1})$ is a $(k-1,\delta, \alpha)$-strainer at $x$. We denote its opposite strainer by $(q_1,\ldots,q_{k-1})$;
    \item $p_k$ is a $(1,\delta)$-strainer at $x$, with $|q_{k-1}x|>|p_kx|$ and $\delta_S(|p_kx|;|q_{k-1}x|)<\delta$;
    \item There exists an opposite strainer $q_k$ of $p_k$ at $x$ such that
    \begin{equation}\label{eq:k_strainer}
      \begin{gathered}
        \tilde{\angle}p_ixp_k>\pi/2 -\alpha,\quad \tilde{\angle}p_ixq_k>\pi/2 -\alpha,\\
        \tilde{\angle}q_ixp_k>\pi/2-\alpha,\quad \tilde{\angle}q_ixq_k>\pi/2-\alpha,
      \end{gathered}
    \end{equation}
    for all $i=1,\ldots, k-1$.
  \end{enumerate}
  In this case, we call the $k$-tuple $(q_1,\ldots,q_k)$ an \emph{opposite strainer} of $(p_1,\ldots,p_k)$ at $x$.
  A point $x\in X$ is called a \emph{$(k,\delta,\alpha)$-strained point} if it admits a $(k,\delta,\alpha)$-strainer at itself.
  A $k$-tuple $(p_1,\ldots,p_k)$ is called a $(k,\delta,\alpha)$-strainer on a subset $U\subset X$ if it is a $(k,\delta,\alpha)$-strainer at each point in $U$.
  In this case, we call the map
  \begin{equation}
    f(\cdot):=\left(\mathsf d_{p_1}(\cdot),\ldots, \mathsf d_{p_k}(\cdot)\right): U\subset X \to \mathbb{R}^k
  \end{equation}
  a \emph{$(k,\delta,\alpha)$-strainer map} on $U$ associated with $(p_1,\ldots,p_k)$, where $\mathsf d_{p}(\cdot):=\mathsf d(p,\cdot)$.
  For notational uniformity, when $k=1$ we use $(k,\delta,\alpha)$-strainer simply to mean a $(1,\delta)$-strainer; in this case, the parameter $\alpha$ plays no role.
\end{definition}
\begin{remark}\label{rmk:asymmetry_strainer_maps}
  We emphasize that the $(k,\delta,\alpha)$-strainers in Definition~\ref{def:k_strainer} are generally non-symmetric, in the sense that the definition only imposes orthogonality of each coordinate of the map $f=(\mathsf d_{p_1},\ldots, \mathsf d_{p_k})$ with respect to the lower-index coordinates.
  This phenomenon, which arises from the asymmetry of angles viewed from a point, is one of the main challenges in non-Riemannian geometry. 
  It appears even in normed spaces, for instance in the context of Birkhoff--James orthogonality in Banach spaces; see Remark~\ref{rmk:asymmetry_angles_viewed_from_fixed_point}.
  Indeed, symmetric Birkhoff--James orthogonality is equivalent to an inner product structure on Banach spaces of dimension at least three.
  We refer to \cite[Example 4.6, Remark 5.8]{han2025structure} for further discussion.
\end{remark}

In the remainder of this subsection, we establish the openness and almost orthogonality of $(k,\delta,\alpha)$-strainers, and show that the degree of orthogonality of strainers is controlled, in a certain sense, by the straightness parameter and the angle-excess parameter in the weak quadruple condition.

\begin{proposition}[Openness and almost orthogonality of strainer maps]\label{prop:almost_orthogonality}
  Let $X$ be an $S$-concave Busemann concave space with $S\geq 1$, and let $\tilde{\varepsilon}\in [0,\pi]$ and $\tilde{\delta}\in (0,\pi]$ be such that $X$ satisfies the $(\tilde{\varepsilon},\tilde{\delta})$-weak quadruple condition.
  Suppose that $\pmb{p}:=(p_1,\ldots,p_k)$ is a $(k,\delta,\alpha)$-strainer at some point $x\in X$ with opposite strainer $\pmb{q}:=(q_1,\ldots,q_k)$, where $\alpha>0$ and $\delta\leq \tilde{\delta}$.
  Then there exists an open neighborhood $U$ of $x$ such that $\pmb{p}$ is a $(k,\delta,\alpha)$-strainer on $U$ with common opposite strainer $\pmb{q}$. Moreover, for every \(1\le i<j\le k\), every \(z\in U\), and every choice of unit-speed geodesics
$\eta_j,\xi_j$ from $z$ to $p_j,q_j$, respectively, one has
  \begin{equation}\label{eq:almost_orthogonality}
    \begin{gathered}
      \left|\angle p_iz\eta_j -\pi/2\right|< 2\delta + \alpha + \tilde{\varepsilon},\quad \left|\angle p_iz\xi_j -\pi/2\right|< 2\delta + \alpha + \tilde{\varepsilon},\\
      \left|\angle q_iz\eta_j -\pi/2\right|< 2\delta + \alpha + \tilde{\varepsilon},\quad \left|\angle q_iz\xi_j -\pi/2\right|< 2\delta + \alpha + \tilde{\varepsilon}.
    \end{gathered}
  \end{equation}
\end{proposition}
\begin{proof}
  First, by continuity of the comparison angles and of the error function $\delta_S$, it is clear that there exists an open neighborhood $U$ of $x$ such that $\pmb{p}$ is a $(k,\delta,\alpha)$-strainer at every point of $U$ with common opposite strainer $\pmb{q}$.
  We now prove the almost orthogonality \eqref{eq:almost_orthogonality}.
  Let $z\in U$ be an arbitrary point, and $\eta_j,\xi_j$ be any unit-speed geodesics from $z$ to $p_j$ and $q_j$, respectively for $1\leq j\leq k$.
  Let $1\leq i<j\leq k$ be fixed.
  For the lower bound of \eqref{eq:almost_orthogonality}, by the monotonicity of the error function $\delta_S$ and the requirement in the definition of strainers that $|p_iz|>|q_iz|>|p_jz|>|q_jz|$, it follows that
  \begin{equation}
    \begin{gathered}
      \delta_S\left(|q_jz|;|p_iz|\right)
      \leq
      \delta_S\left(|p_jz|;|p_iz|\right)
      \leq
      \delta_S\left(|p_jz|; |q_iz|\right)
      <
      \delta,\\
      \delta_S\left(|q_jz|;|q_iz|\right)
      \leq
      \delta_S\left(|q_jz|;|p_jz|\right)< \delta.
    \end{gathered}
  \end{equation}
  Thus, by the almost comparison inequality (Lemma~\ref{lem:angle_viewed_from_fixed_point}~\ref{item:almost_comparison}), we obtain that
  \begin{equation}\label{eq:almost_orthogonality_1}
    \begin{gathered}
    \angle p_i z\eta_j \geq \tilde{\angle}p_izp_j - \delta_S(|p_jz|;|p_iz|)
    > 
    \frac{\pi}{2} -\alpha - \delta,\\
    \angle p_i z\xi_j \geq \tilde{\angle}p_izq_j - \delta_S(|q_jz|;|p_iz|)
    >
    \frac{\pi}{2} - \alpha -\delta.
  \end{gathered}
  \end{equation}
  By the same argument, we also obtain that $\angle q_iz\eta_j>\pi/2-\alpha-\delta$ and $\angle q_iz\xi_j>\pi/2-\alpha-\delta$.

  For the upper bound of \eqref{eq:almost_orthogonality}, pick up an arbitrary $t\in (0,l(\eta_j))$. 
  Since $\delta\leq \tilde{\delta}$, by applying the $(\tilde{\varepsilon}, \tilde{\delta})$-weak quadruple condition to the quadruple $(z;p_i,q_i, \eta_j(t))$, we have that
  \begin{equation}
    \tilde{\angle}p_iz \eta_j(t) + \tilde{\angle}q_iz\eta_j(t)
    \leq
    2\pi + \tilde{\varepsilon} - \tilde{\angle}p_izq_i
    \leq
    \pi + \tilde{\varepsilon} + \delta.
  \end{equation}
  Together with the definition of angles viewed from a point \eqref{eq:angle_viewed_from_fixed_point} and the lower bound \eqref{eq:almost_orthogonality_1}, we obtain that
  \begin{equation}
    \begin{gathered}
      \angle p_iz\eta_j
      \leq
      \pi + \tilde{\varepsilon} + \delta - \angle q_i z\eta_j
      \leq
      \frac{\pi}{2} + \tilde{\varepsilon} + 2\delta + \alpha,\\
      \angle q_iz\eta_j
      \leq
      \pi + \tilde{\varepsilon} + \delta - \angle p_iz\eta_j
      \leq
      \frac{\pi}{2} + \tilde{\varepsilon} + 2\delta + \alpha.
    \end{gathered}
  \end{equation}
  By applying the same argument to the quadruple $(z;p_i,q_i,\xi_j(s))$ for $s\in (0,l(\xi_j))$, we also obtain the same upper bound for $\angle p_iz\xi_j$ and $\angle q_iz\xi_j$.
  Thus, our claim follows.

\end{proof}

\subsection{Openness and bi-Lipschitz property of strainer maps}\label{subsect:openness_bilipschitz}

\noindent
We first briefly recall the definition of $\varepsilon$-openness; see \cite{lytchak2006open, lytchak2019geod} for more details.
\begin{definition}[$\varepsilon$-open maps]\label{def:epsilon_open_map}
    A Lipschitz map $F: U\to Y$ from an open subset $U\subset X$ to a metric space $Y$ is said to be \emph{$\varepsilon$-open} if for any point $x\in U$, there exists $r>0$ such that the closed ball $\bar{B}(x,\varepsilon^{-1}r)\subset U$ is complete, and for any $v\in B(F(x),r)\subset Y$, there exists a point $y\in U$ such that $F(y)=v$ and $\varepsilon|yx|\leq |F(x)v|$.
    In particular, for any $s\in (0,r]$, we have the inclusion $B(F(x),s)\subset F(B(x,\varepsilon^{-1}s))$.
\end{definition}

We now prove that $(k,\delta,\alpha)$-strainer maps are open whenever the combined size of the parameters $\delta$ and $\alpha$ of strainers, together with the angle excess $\tilde{\varepsilon}$ of the weak quadruple condition, is small.
Our argument follows the same strategy as in \cite{fujioka2025top} and \cite{han2025structure}: we equip the target space with an anisotropic $\ell_1$-norm to compensate for the asymmetry of strainer maps and apply a criterion for $\varepsilon$-open maps from a locally complete metric space to a geodesic space; see \cite[Lemma 5.11]{han2025structure}.
The proof closely parallels the arguments in our previous work \cite[Proposition 5.13]{han2025structure}, which rely only on the almost orthogonality of strainer maps established in \cite[Lemma 5.7]{han2025structure}, rather than local semi-convexity itself.
We therefore provide only a brief sketch here, replacing that lemma with the almost orthogonality (Proposition~\ref{prop:almost_orthogonality}) established in the present paper.

\begin{proposition}[$\varepsilon$-openness]\label{prop:epsilon_openness_strainer_maps}
  Let $X$ be an $S$-concave Busemann concave space with $S\geq 1$, and $k\in \mathbb{N}$.
  There exists a constant $\delta_k:=k^{-1}2^{-2k}$ such that the following holds: let $\tilde{\varepsilon}\in [0,\delta_k),\tilde{\delta}\in (0,\pi]$ be given constants such that $X$ satisfies the $(\tilde{\varepsilon},\tilde{\delta})$-weak quadruple condition.
  Let $f:U\to \mathbb{R}^k$ be a $(k,\delta,\alpha)$-strainer map from an open subset $U\subset X$ to $\mathbb{R}^k$, where $\mathbb{R}^k$ is equipped with the $\ell_1$-norm, with the parameters $\delta,\alpha$ satisfying $\delta\leq \tilde{\delta}$ and $2\delta + \alpha + \tilde{\varepsilon} \leq \delta_k$.
  Then $f$ is $\varepsilon_k(\tilde{\varepsilon},\delta,\alpha):=(1-2\delta-\alpha-\tilde{\varepsilon})/4^{k-1}$-open.
\end{proposition}
\begin{proof}
  Let $(p_1,\ldots,p_k)$ be the $(k,\delta,\alpha)$-strainer associated with $f$.
  For each $i=1,\ldots,k$, set $f_i:=\mathsf d_{p_i}$ and $f_{[i]}:=(f_1,\ldots,f_i)$.
  For simplicity, we write $\varepsilon_k:=\varepsilon_k(\tilde{\varepsilon},\delta,\alpha)$ throughout the rest of the proof.
  We show the claim by induction on $i$.
  \begin{enumerate}[fullwidth, label=\textbf{Step \arabic*}:]
    \item We first show $f_1$ is $\varepsilon_1$-open.
    Note that the argument in Step~1 of \cite[Proposition 5.13]{han2025structure} uses only the almost comparison inequality and the first variation formula, namely Lemma~\ref{lem:angle_viewed_from_fixed_point}~\ref{item:almost_comparison} and~\ref{item:first_variation}, and does not rely on local semi-convexity.
    Therefore, the same argument shows that $f_1$ is $(1-2\delta)$-open as a map from $U$ to $\mathbb{R}$.
    Take $\varepsilon_1:=1-2\delta-\alpha-\tilde{\varepsilon}$ for $\delta\leq \tilde{\delta}$ and $2\delta + \alpha + \tilde{\varepsilon} \leq \delta_1:=1/4$.
    Since $\varepsilon_1\leq 1-2\delta$, Definition~\ref{def:epsilon_open_map} implies that $f_1$ is also $\varepsilon_1$-open.

    \item Suppose that the $(k-1, \delta,\alpha)$-strainer map $f_{[k-1]}$ is $\varepsilon_{k-1}$-open from $U$ to $\mathbb{R}^{k-1}$ equipped with the $\ell_1$-norm, where $\varepsilon_{k-1}=(1-2\delta-\alpha-\tilde{\varepsilon})/4^{k-2}$.
    Let $\eta_k,\xi_k$ be any unit-speed geodesics from $x$ to $p_k$ and to $q_k$, respectively.
    For $v\in \mathbb{R}^k$, let the anisotropic norm $\|\cdot\|$ on $\mathbb{R}^k$ be defined as:
    \begin{equation}\label{eq:anisotropic_norm}
      \|v\|:=\left\|v_{[k-1]}\right\|_1 + \frac{\varepsilon_{k-1}}{2}\left|v_k\right|,\quad v=(v_1,\ldots,v_k)\in \mathbb{R}^k,
    \end{equation}
    where $v_{[k-1]}$ is the first $k-1$ coordinates of $v$, and $\|\cdot\|_1$ is the usual $\ell_1$-norm on $\mathbb{R}^{k-1}$.
    Let $x\in U$ be fixed.
    By the criterion for $\varepsilon$-open maps (see \cite[Lemma 5.11]{han2025structure}), it suffices to show that for any $v\in \mathbb{R}^k\setminus \{f(x)\}$ sufficiently close to $f(x)$, we can find $y\in U$ such that
    \begin{equation}
      \left|f(y)v\right| - \left|f(x)v \right| \leq -\varepsilon'_k |xy|,
    \end{equation}
    for some $\varepsilon'_k>\varepsilon_k$.
    As in \cite[Proposition 5.13]{han2025structure}, we divide the argument into two cases.
    We emphasize that both cases in \cite[Proposition 5.13]{han2025structure} rely only on the almost orthogonality of strainer maps, and no local semi-convexity is used there.
    We therefore omit some repeated computations.
    \begin{enumerate}[fullwidth, label=\textbf{Case \arabic*:}]
      \item $f_{[k-1]}\neq v_{[k-1]}$.
      By the same argument as in Case~1 in Step~2 of \cite[Proposition 5.13]{han2025structure}, we can find $y\in U$ such that
      \begin{equation}
        \|f(y)-v\| - \|f(x)- v\| \leq -\frac{\varepsilon_{k-1}}{2}|xy|.
      \end{equation}

      \item $f_k\neq v_k$.
      Following the same argument as Step~1 of \cite[Proposition 5.13]{han2025structure}, we can find $y\in U$ lying on one of the geodesics $\eta_k$ and $\xi_k$, such that
      \begin{equation}\label{eq:openness_strainer_maps_2_1}
        f_k(y)=v_k,\quad (1-2\delta -\alpha - \tilde{\varepsilon})|xy|\leq
        |f_k(x)-v_k|.
      \end{equation}
      Note that from \eqref{eq:openness_strainer_maps_2_1}, we see that $y$ is close to $x$ whenever $v$ is close to $f(x)$.
      From the almost orthogonality of strainer maps (Proposition~\ref{prop:almost_orthogonality}), we have that
      \begin{equation}\label{eq:openness_strainer_maps_2_2}
        \left|\cos \angle p_ix\eta_k\right|
        <
        2\delta + \alpha + \tilde{\varepsilon},\quad
        \left|\cos \angle p_ix\xi_k\right|
        <
        2\delta + \alpha + \tilde{\varepsilon},\quad \text{for } i=1,\ldots,k-1.
      \end{equation}
      The inequality \eqref{eq:openness_strainer_maps_2_2}, together with the first variation formula (Lemma~\ref{lem:angle_viewed_from_fixed_point}~\ref{item:first_variation}), implies that $|f_i(y)-f_i(x)|\leq (2\delta + \alpha + \tilde{\varepsilon})|xy|$ for $i=1,\ldots,k-1$, whenever $v$ is close to $f(x)$.
      Thus, following exactly the same argument as in Case~2 in Step~2 of \cite[Proposition 5.13]{han2025structure}, we obtain that
      \begin{equation}
        \| f(y)-v\| - \|f(x) -v\|
        \leq
        -\left(\frac{\varepsilon_{k-1}}{2}(1-2\delta-\alpha -\tilde{\varepsilon}) - (2\delta + \alpha + \tilde{\varepsilon})(k-1)\right)|xy|.
      \end{equation}
      Let $\delta_k:=k^{-1}2^{-2k}$.
      One can readily check that for $\delta,\alpha$ satisfying $\delta\leq \tilde{\delta}$ and $\tilde{\varepsilon} + 2\delta + \alpha\leq \delta_k$, it holds that
      \begin{equation}
        \| f(y) -v\| - \| f(x) - v\|
        \leq
        -\frac{5\varepsilon_{k-1}}{16}|xy|
        <
        -\frac{\varepsilon_{k-1}}{4}|xy|.
      \end{equation}
    \end{enumerate}

    Combining Cases 1 and 2, we obtain
    $\|f(y)-v\|-\|f(x)-v\|\leq -\varepsilon'_k|xy|$ whenever $v$ is close to $f(x)$, where $\varepsilon'_k:=\varepsilon_{k-1}/2$.
    By the criterion for $\varepsilon_k$-open maps (see \cite[Lemma 5.11]{han2025structure}) and the fact that the $\ell_1$-norm dominates the anisotropic norm $\|\cdot\|$ defined in \eqref{eq:anisotropic_norm}, we obtain that $f$ is $\varepsilon_{k-1}/4$-open as a map from $U$ to $\mathbb{R}^{k}$ equipped with the $\ell_1$-norm.
    By induction, we conclude that $f$ is $\varepsilon_k$-open from $U$ to $\mathbb{R}^k$ equipped with the $\ell_1$-norm, with $\varepsilon_k:=(1-2\delta-\alpha-\tilde{\varepsilon})/4^{k-1}$ for $\delta\leq \tilde{\delta} $ and $\tilde{\varepsilon} + 2\delta + \alpha\leq \delta_k$.
  \end{enumerate}
\end{proof}

The following corollary is an immediate consequence of Propositions~\ref{prop:almost_orthogonality} and~\ref{prop:epsilon_openness_strainer_maps}.

\begin{corollary}\label{cor:epsilon_openness_strainer_maps}
    Let $X$ be an $S$-concave Busemann concave space with $S\geq 1$.
    Let $k\in \mathbb{N}$ and $\tilde{\varepsilon}\in [0,\delta_k)$ and $\tilde{\delta}\in (0,\pi]$ be given constants such that $X$ satisfies the $(\tilde{\varepsilon},\tilde{\delta})$-weak quadruple condition.
    Suppose that $(p_1,\ldots,p_k)$ is a $(k,\delta,\alpha)$-strainer at $x\in X$ with an opposite strainer $(q_1,\ldots,q_k)$, where the parameters satisfy $\delta\leq \tilde{\delta}$ and $2\delta+\alpha+\tilde{\varepsilon}\leq \delta_k$.
    Then there exists an open neighborhood $U$ of $x$ such that the associated $(k,\delta,\alpha)$-strainer map $f:U\to \mathbb{R}^k$, where $\mathbb{R}^k$ is equipped with the standard Euclidean norm, is $\sqrt{k}$-Lipschitz and $\bar{\varepsilon}_k$-open, where $\bar{\varepsilon}_k:= \varepsilon_k(\tilde{\varepsilon},\delta,\alpha)/\sqrt{k}$, and $\delta_k,\varepsilon_k$ are given in Proposition~\ref{prop:epsilon_openness_strainer_maps}.
\end{corollary}

Next, we establish the bi-Lipschitz property of strainer maps when the parameters $\delta$ and $\alpha$ of strainers, together with the angle-excess parameter $\tilde{\varepsilon}$ of the weak quadruple condition, are sufficiently small.
The key step is to establish the injectivity of strainer maps on an open subset where higher-dimensional strainers do not exist.
In the absence of local semi-convexity, we cannot apply the argument in \cite[Proposition 5.15]{han2025structure}, whose proof relies on the local geometry of locally semi-convex, $S$-concave spaces, where the sum of adjacent angles viewed from a common point equals $\pi$ (see \cite[Lemma 4.8 and 4.11]{han2025structure}).
Instead, we follow an approach similar to that of \cite[Lemma 5.6 and Corollary 5.7]{burago1992ad}, based on the weak quadruple condition, the upper bound for the sum of adjacent angles viewed from a common point (Lemma~\ref{lem:angle_viewed_from_fixed_point}~\ref{item:sum_adjacent_angles}), and the distance-angle estimates from the configuration of strainers.

\begin{proposition}[Bi-Lipschitz homeomorphism]\label{prop:bi_Lipschitz}
    Let $X$ be an $S$-concave Busemann concave space with $S\geq 1$.
    Let $k\in \mathbb{N}$, and $\tilde{\varepsilon}\in [0,\delta_k), \tilde{\delta}\in (0,\pi]$ be given constants such that $X$ satisfies the $(\tilde{\varepsilon}, \tilde{\delta})$-weak quadruple condition.
    Let $\pmb{p}=(p_1,\ldots,p_k)$ be a $(k,\delta,\alpha)$-strainer at $x_0\in X$ with $\delta\leq \tilde{\delta}$ and $2\delta + \alpha + \tilde{\varepsilon} \leq \delta_k$, with an opposite strainer $\pmb{q}=(q_1,\ldots,q_k)$\footnote{We take $\alpha=0$ in the case $k=1$.}.
    Suppose that there is an open neighborhood $V$ of $x_0$ such that no point in $V$ admits a $(k+1, \delta, \max\{5\delta+2\tilde{\varepsilon},\alpha\})$-strainer.
    Then the associated strainer map $f=(\mathsf d_{p_1},\ldots,\mathsf d_{p_k})$ is a bi-Lipschitz homeomorphism from some open neighborhood of $x_0$ to a domain in $\mathbb{R}^k$.
\end{proposition}
\begin{proof}
  It suffices to show that $f$ is injective in some neighborhood of $x_0$ (see \cite[Theorem 5.30]{fujioka2025top}).
  By Proposition~\ref{prop:almost_orthogonality}, we can choose $r_0>0$ small such that $B(x_0, 2r_0)\subset V$, $\pmb{p}$ is a $(k,\delta,\alpha)$-strainer on $B(x_0,2r_0)$ with the common opposite strainer $\pmb{q}$, and
  \begin{equation}\label{eq:bi_Lipschitz_00}
    \max\{4r_0/|q_kx_0|, \delta_S(4r_0;|q_kx_0|)\}<0.01\delta.
  \end{equation}

  We show by contradiction that $f$ is injective on $U:=B(x_0,r_0)$.
  Suppose the claim does not hold.
  Then there exist two distinct points $x,y\in U$ such that $f(x)=f(y)$.
  Let $\gamma:[0,l]\to X$ be a unit-speed geodesic from $x$ to $y$, where $l:=l(\gamma)\leq 2r_0$.
  We show that there exists $z\in \gamma$ close to $x$, such that $(p_1,\ldots,p_k, y)$ is a $(k+1, \delta, \max\{5\delta+2\tilde{\varepsilon}, \alpha\})$-strainer at $z$, which contradicts the assumption.
  \begin{enumerate}[fullwidth, label=\textbf{Step \arabic*: }]
    \item Choose $z\in \gamma$ such that 
    \begin{equation}\label{eq:bi_Lipschitz_001}
        \max\{|xz|/|yz|, \delta_S(|xz|;|yz|)\}<\delta/8.
    \end{equation}
    Note that the assumptions of the lemma imply $\delta<1/2$.
    It is also clear that $\tilde{\angle}yzx =\pi>\pi -\delta$.
    Therefore, by Definition~\ref{def:1_strainer}, it follows that $z$ is a $(1,\delta)$-strained point with the strainer pair $(y,x)$.
    It remains to check the almost orthogonality of $\tilde{\angle}p_izy, \tilde{\angle}q_izy, \tilde{\angle}p_izx$ and $\tilde{\angle}q_izx$ for $i=1,\ldots,k$.

    \item In this step, we show an auxiliary result that for any $z'\in \gamma$, the strainer pair $(p_i,q_i)$ together with $y$ and $x$ satisfies the following angle-sum bounds:
    \begin{equation}\label{eq:bi_Lipschitz_0}
      \left|\tilde{\angle}p_iz'y + \tilde{\angle}q_iz'y -\pi\right|< 2\delta + \tilde{\varepsilon},\quad
      \left|\tilde{\angle}p_iz'x + \tilde{\angle}q_iz'x - \pi\right|<2\delta + \tilde{\varepsilon},
    \end{equation}
    for all $i=1,\ldots,k$.
    Let $i\in \{1,\ldots,k\}$ be fixed. 
    For notational simplicity, we denote $p:=p_i$ and $q:=q_i$.
    For the upper bound of \eqref{eq:bi_Lipschitz_0}, by our choice of $r_0$, it follows that $p$ is a $(1,\delta)$-strainer at $z'$ with the opposite strainer $q$.
    Thus, by applying the $(\tilde{\varepsilon}, \tilde{\delta})$-weak quadruple condition to the quadruple $(z'; p,q,y)$, we obtain that
    \begin{equation}\label{eq:bi_Lipschitz_1}
      \tilde{\angle}pz'y + \tilde{\angle}qz'y
      \leq
      2\pi + \tilde{\varepsilon} - \tilde{\angle}pz'q
      <
      2\pi + \tilde{\varepsilon} - (\pi-\delta)
      <
      \pi + \tilde{\varepsilon} + \delta.
    \end{equation}
    For the lower bound of \eqref{eq:bi_Lipschitz_0}, from the inequality \eqref{eq:bi_Lipschitz_00}, one can readily check that
    \begin{equation}
      \max\{|z'y|, |z'x|\}\leq 0.01\delta\min\{|pz'|,|qz'|\},
    \end{equation}
    from which we obtain that
    \begin{equation}\label{eq:bi_Lipschitz_2}
      \tilde{\angle}pz' y + \tilde{\angle}pyz'> \pi -\delta/2,\quad
      \tilde{\angle}qz'y + \tilde{\angle}qyz' > \pi -\delta/2.
    \end{equation}
    Summing up the inequalities in \eqref{eq:bi_Lipschitz_2}, we obtain that
    \begin{equation}\label{eq:bi_Lipschitz_3}
      \tilde{\angle}pz' y + \tilde{\angle}pyz' + \tilde{\angle}qz'y + \tilde{\angle}qyz'>2\pi - \delta.
    \end{equation}
    On the other hand, by applying the $(\tilde{\varepsilon}, \tilde{\delta})$-weak quadruple condition to the quadruple $(y;p,q,z')$, we have that
    \begin{equation}
      \tilde{\angle}pyz' + \tilde{\angle}qyz'
      \leq
      2\pi + \tilde{\varepsilon} - \tilde{\angle}pyq
      <
      \pi + \tilde{\varepsilon} + \delta.
    \end{equation}
    Plugging this inequality into the inequality \eqref{eq:bi_Lipschitz_3}, we obtain that
    \begin{equation}\label{eq:bi_Lipschitz_4}
      \tilde{\angle}pz' y + \tilde{\angle}qz'y
      >
      2\pi - \delta - \left(\tilde{\angle}pyz' + \tilde{\angle}qyz'\right)
      >
      \pi - 2\delta - \tilde{\varepsilon}.
    \end{equation}
    Together with \eqref{eq:bi_Lipschitz_1} and \eqref{eq:bi_Lipschitz_4}, we obtain the first inequality in \eqref{eq:bi_Lipschitz_0}.
    By the same argument, we can obtain the second inequality in \eqref{eq:bi_Lipschitz_0}, and therefore our claim follows.

    \item
    We now show the desired almost orthogonality of $\tilde{\angle}pzy, \tilde{\angle}qzy, \tilde{\angle}pzx$ and $\tilde{\angle}qzx$.
    For the first two comparison angles, let $\eta$ and $\xi$ denote the unit-speed geodesic segments of $\gamma$ from $z$ to $y$ and to $x$, respectively.
    By applying the Euclidean law of cosines to the comparison triangle $\tilde{\Delta}pzy$, it follows that
    \begin{multline}
      \left|\cos\tilde{\angle}pzy\right|
      =
      \left|\frac{|pz|^2+|zy|^2-|py|^2}{2|pz||zy|}\right|
      \leq
      \frac{|zy|}{2|pz|} + \frac{(|pz|+|px|)\big||pz|-|px|\big|}{2|pz||zy|}\\
      \leq
      \frac{r_0}{2(1-0.01\delta)|px|} + \frac{(2+0.01\delta)|px||xz|}{2(1-0.01\delta)|px||zy|}
      \leq
      0.1\delta + 1.1\,\frac{|xz|}{|zy|},
    \end{multline}
    where we use the hypothesis $|px|=|py|$ in the first inequality, and the last two inequalities follow from the choice of $r_0$ and \eqref{eq:bi_Lipschitz_00}.
    From our choice of $z\in \gamma$, namely \eqref{eq:bi_Lipschitz_001}, we obtain that $|\cos \tilde{\angle}pzy|<\delta/4$, which implies that $|\tilde{\angle}pzy - \pi/2|<\delta/2$.
    By taking $z':=z$ in the first inequality in \eqref{eq:bi_Lipschitz_0}, we obtain that $|\tilde{\angle}qzy-\pi/2|< 5\delta/2 + \tilde{\varepsilon}$.
    For the latter two comparison angles, note that from the almost comparison inequality, the sum of adjacent angles viewed from a common point (Lemma~\ref{lem:angle_viewed_from_fixed_point}~\ref{item:almost_comparison},~\ref{item:sum_adjacent_angles}) and the choice of $r_0$, we have that
    \begin{equation}
      \tilde{\angle}pzy + \tilde{\angle}pzx
      \leq
      \angle pz \eta + \angle pz\xi + \delta_S(|zy|;|pz|) + \delta_S(|zx|; |pz|)
      <
      \pi + \delta/2.
    \end{equation}
    Plugging the almost orthogonality $|\tilde{\angle}pzy-\pi/2|<\delta/2$ into the above inequality, we have that
    \begin{equation}
      \tilde{\angle}pzx
      \leq
      \pi + \delta/2 - \tilde{\angle}pzy
      <
      \pi/2 +\delta.
    \end{equation}
    The same argument, together with the almost orthogonality that $|\tilde{\angle}qzy-\pi/2|<5\delta/2 + \tilde{\varepsilon}$, implies that
    \begin{equation}
      \tilde{\angle}qzx
      \leq
      \pi + \delta/2 -\tilde{\angle}qzy
      <
      \pi/2 + 3\delta + \tilde{\varepsilon}.
    \end{equation}
    By taking $z':=z$ in the second inequality in \eqref{eq:bi_Lipschitz_0}, we obtain that 
    \begin{equation}
      \begin{gathered}
        \tilde{\angle}pzx
        >
        \pi - 2\delta - \tilde{\varepsilon} - \tilde{\angle}qzx
        >
        \frac{\pi}{2} - 5\delta - 2\tilde{\varepsilon},\\
        \tilde{\angle}qzx
        >
        \pi-2\delta - \tilde{\varepsilon}- \tilde{\angle}pzx
        >
        \frac{\pi}{2} - 3\delta - \tilde{\varepsilon}.
      \end{gathered}
    \end{equation}
  \end{enumerate}

  To sum up, we have shown that $y$ is a $(1,\delta)$-strainer at $z$ with $x$ as the opposite strainer, and that $\tilde{\angle}p_izy,\tilde{\angle}p_izx, \tilde{\angle}q_izy, \tilde{\angle}q_izx$ are greater than $\pi/2- 5\delta -2\tilde{\varepsilon}$ for $i=1,\ldots, k$.
  Thus, by the definition of strainers, it follows that $(p_1,\ldots,p_k,y)$ is a $(k+1, \delta, \max\{5\delta+2\tilde{\varepsilon}, \alpha\})$-strainer at $z\in B(x_0,2r_0)\subset V$.
  But this contradicts our assumption that no point in $V$ admits a $(k+1, \delta, \max\{5\delta + 2\tilde{\varepsilon}, \alpha\})$-strainer. 
  Hence, $f$ is injective on $U$, which completes the proof.
\end{proof}

\subsection{Self-improvement of strainers}\label{subsect:self-improving}

\noindent
We establish a self-improvement property of $(k,\delta,\alpha)$-strainers on $S$-concave Busemann concave spaces. Any $(k,\delta,\alpha)$-strained point whose parameters, together with the angle excess $\tilde{\varepsilon}$ of the weak quadruple condition, are sufficiently small has arbitrarily close $(k,\rho,\rho)$-strained points for every $0<\rho<1/2$.
This property has been established in \cite[Lemma 5.9]{burago1992ad} for Alexandrov spaces with curvature bounded below and in \cite[Lemma 5.16]{han2025structure} for locally semi-convex, $S$-concave Busemann concave spaces.

In \cite{han2025structure}, we used a new `dequeuing-and-enqueuing' approach to handle the asymmetry of strainers in locally semi-convex, $S$-concave Busemann concave spaces.
Although the proof of the self-improvement property here follows the same general strategy, an additional difficulty arises in the present setting: the $(\tilde{\varepsilon}, \tilde{\delta})$-weak quadruple condition can improve a given strainer only to one whose almost orthogonality is no better than the angle-excess parameter $\tilde{\varepsilon}$.

To overcome this difficulty, we adopt a strengthened `dequeuing-and-enqueuing' procedure based on the weak quadruple condition with varying parameters.
Starting from a $(k,\delta,\alpha)$-strainer $(p_1,\ldots,p_k)$, we first dequeue the first strainer pair $(p_1,q_1)$ and enqueue the new strainer pair $(p'_1,q'_1)$ from a straightening procedure, obtaining the $k$-tuple $(p_2,\ldots, p_k,p'_1)$.
This tuple is a mixed strainer: it remains a $k$-dimensional strainer with controlled parameters, while the new last entry $p'_1$, with its opposite strainer $q'_1$, forms a 1-strainer pair with the improved `straightness' parameter.
This improvement and Proposition~\ref{lemma:verification_WQC_II} allow us, in the next step, to apply the weak quadruple condition with a smaller angle-excess parameter $\varepsilon'$. 
This yields better almost orthogonality among newly created strainer pairs, while the weak quadruple condition with the original parameters keeps the almost orthogonality among old and new strainer pairs under control.

After $k$ steps of the strengthened `dequeuing-and-enqueuing' process, the resulting $k$-tuple gives the desired strainer, with improved straightness and almost-orthogonality properties.

\begin{lemma}[Self-improvement of strainers]\label{lemma:almost_self_improvement}
    Let $X$ be an $S$-concave Busemann concave space with $S\geq 1$.
    Let $k\in \mathbb{N}$, and $\tilde{\varepsilon}\in [0,\delta_k/3), \tilde{\delta}\in (0,\pi]$ be given constants such that $X$ satisfies the $(\tilde{\varepsilon}, \tilde{\delta})$-weak quadruple condition.
    Suppose $x\in X$ is a $(k,\delta,\alpha)$-strained point with $\delta\leq \tilde{\delta}$ and $7\delta+\alpha + 3\tilde{\varepsilon}\leq \delta_k$.
    Then, for every open neighborhood $U$ of $x$, every $0<\varepsilon'<\tilde{\varepsilon}$, and every
    \[
      0<\delta'<\min\{\delta,\bar{\delta}(\varepsilon',S)\},
    \]
    where $\bar{\delta}(\varepsilon',S)$ is the constant supplied by Proposition~\ref{lemma:verification_WQC_II}, the neighborhood $U$ contains a $(k,\delta', 5\delta'+ 2\varepsilon')$-strained point.
    In particular, every open neighborhood of $x$ contains a $(k,\rho,\rho)$-strained point for every $0<\rho <1/2$.
\end{lemma}
\begin{proof}
  Let $\pmb{p}=(p_1,\ldots,p_k)$ be a $(k,\delta,\alpha)$-strainer at $x$ with the opposite strainer $\pmb{q}=(q_1,\ldots,q_k)$, and let $f$ be the associated strainer map.
  Let $U\subset X$ be an arbitrary open neighborhood of $x$, and let $0<\varepsilon'<\tilde{\varepsilon}$ be an arbitrary positive number.
  By Proposition~\ref{lemma:verification_WQC_II}, we can find $\bar{\delta}:=\bar{\delta}(\varepsilon',S)>0$ such that $X$ satisfies the $(\varepsilon', \delta')$-weak quadruple condition for any $\delta'\in (0,\bar{\delta}]$.
  Let $0<\delta'<\min\{\delta, \bar{\delta}\}$ be an arbitrary number.
  In the following, we show that there exists a $(k, \delta', 5\delta' + 2\varepsilon')$-strained point in $U$.

  For the case $k=1$, the claim is trivial, since we can choose a point $z\in U$ on a unit-speed geodesic from $x$ to $p_1$ sufficiently close to $x$, such that $z$ is a $(1,\delta')$-strained point with the strainer pair $(p_1,x)$.
  Therefore, we can assume $k\geq 2$.
  We will dequeue and enqueue the strainer points $p_i$ from $i = 1$ to $k$ inductively to construct the desired strained point.

  \begin{enumerate}[fullwidth, label=\textbf{Step \arabic*:}]
    \item We first dequeue and enqueue the strainer pair $(p_1,q_1)$.
    Since $\delta\leq \tilde{\delta}$ and $7\delta+\alpha+3\tilde{\varepsilon}\leq \delta_k$ from the assumptions of the lemma, Corollary~\ref{cor:epsilon_openness_strainer_maps} gives a radius $r_0>0$ small enough such that $U_1:=B(x,2r_0)\subset U$, the associated strainer map $f:V_1=B(x,r_0)\to \mathbb{R}^k$ is $\bar{\varepsilon}_k$-open with the common opposite strainer $\pmb{q}$, and
    \begin{equation}
      \max\{4r_0/|q_kx|, \delta_S(4r_0;|q_kx|)\}<0.01\delta'.
    \end{equation}
    By the openness of $f$, we can find $y\in V_1\setminus \{x\}$ such that $|p_ix|=|p_iy|$ for $i=2,\ldots, k$.
    Let $\xi$ be a unit-speed geodesic from $x$ to $y$.
    Since $|p_i x|=|p_i y|$ for $i=2,\ldots,k$, the same argument as in Steps~2--3 of the proof of Proposition~\ref{prop:bi_Lipschitz}, together with the $(\tilde{\varepsilon}, \tilde{\delta})$-weak quadruple condition and the choice of $r_0$, allows us to choose $z_1\in \xi \cap V_1$ sufficiently close to $x$ such that the new $k$-tuple $(p_2,\ldots,p_k,y)$ is a $(k,\delta,\max\{\alpha,5\delta + 2\tilde{\varepsilon}\})$-strainer at $z_1$.
    Moreover, its first $(k-1)$ entries form a $(k-1,\delta,\alpha)$-strainer at $z_1$ with opposite strainer $(q_1,\ldots,q_{k-1})$, while its last entry $y$ is a $(1,\delta')$-strainer at $z_1$ with the opposite strainer $x$.
    We denote by $p'_1:=y$ and $q'_1:=x$.

    \item Suppose that, for some $2\leq j \leq k$, we have a point $z_{j-1}\in U_1$ and a $k$-tuple
    \begin{equation}
        \pmb{p}_{j-1}:=(p_j,\ldots,p_k,p'_1,\ldots, p'_{j-1})
    \end{equation}
    satisfying the following conditions:
    \begin{enumerate}
        \item $\pmb{p}_{j-1}$ is a $(k,\delta, \max\{\alpha, 5\delta + 2\tilde{\varepsilon}\})$-strainer at $z_{j-1}$, with the opposite strainer
        \begin{equation}
          \pmb{q}_{j-1}:=(q_j,\ldots,q_k,q'_1,\ldots, q'_{j-1});
        \end{equation}
        
        \item The first $k-j+1$ entries $(p_j,\ldots, p_k)$ form a $(k-j+1,\delta,\alpha)$-strainer at $z_{j-1}$, with the opposite strainer $(q_j,\ldots,q_k)$;
        \item The last $j-1$ entries $(p'_1,\ldots,p'_{j-1})$ form a $(j-1, \delta', 5\delta' + 2\varepsilon')$-strainer at $z_{j-1}$ with the opposite strainer $(q'_1,\ldots,q'_{j-1})$;
        \item $|xz_{j-1}|<\sum_{i=1}^{j-1}r_0/2^{i-1}$.
    \end{enumerate}
    Let $f_{j-1}$ be the associated strainer map of $\pmb{p}_{j-1}$.
    For $\pmb{p}_{j-1}$, by our assumption on $\tilde{\varepsilon},\alpha,\delta$, we have that
    \begin{equation}
      2\delta + \max\{\alpha, 5\delta + 2\tilde{\varepsilon}\} + \tilde{\varepsilon}
      \leq
      7\delta + \alpha + 3\tilde{\varepsilon}\leq \delta_k.
    \end{equation} 
    Therefore, by the openness of strainers and strainer maps (Proposition~\ref{prop:almost_orthogonality} and Corollary~\ref{cor:epsilon_openness_strainer_maps}), we can choose a radius $r_j<r_0/2^j$ small enough such that all assumptions (a)--(c) above hold for every point in $U_j:=B(z_{j-1},2r_j)$ in place of the point $z_{j-1}$, and the strainer map $f_{j-1}$ is $\bar{\varepsilon}_k$-open on $V_j:=B(z_{j-1}, r_j)$.
    We also choose $r_j$ so that
    \begin{equation}
      \max\{4r_{j}/|q'_{j-1}z_{j-1}|, \delta_S(4r_{j};|q'_{j-1}z_{j-1}|)\}<0.01\delta'.
    \end{equation}
    By the openness of $f_{j-1}$ on $V_j$, we can find $p'_j\in V_j$ such that all entries of $f_{j-1}(p'_j)$, except the first one, are the same as those of $f_{j-1}(z_{j-1})$.
    That is, $|p_iz_{j-1}|=|p_ip'_j|$ for $i=j+1,\ldots, k$ and $|p'_{i}z_{j-1}|=|p'_{i}p'_j|$ for $i=1,\ldots, j-1$.
    Let $\xi_j$ be a unit-speed geodesic from $q'_j:=z_{j-1}$ to $p'_j$.
    We choose $z_j\in \xi_j\cap V_j$ sufficiently close to $z_{j-1}$ such that
    \begin{equation}
      \max\{|q'_jz_j|/|q'_jp'_j|, \delta_S(|q'_jz_j|;|q'_jp'_j|)\}<\delta'/8.
    \end{equation}

    The new $k$-tuple $\pmb{p}_j:=(p_{j+1},\ldots,p_k,p'_1,\ldots, p'_j)$ forms a mixed strainer at $z_j$.
    Indeed, for the strainer pairs $(p_i,q_i)$ and $(p'_j,q'_j)$ at the base point $z_j$ with $i=j+1,\ldots,k$, assumption (b) above and the same argument as in Step~2--3 of the proof of Proposition~\ref{prop:bi_Lipschitz}, together with the $(\tilde{\varepsilon}, \tilde{\delta})$-weak quadruple condition, implies that
    \begin{equation}\label{eq:almost_self_improvement_2_1}
        \tilde{\angle}p_iz_jp'_j,\: \tilde{\angle}p_iz_jq'_j,\: \tilde{\angle}q_iz_jp'_j,\: \tilde{\angle}q_iz_jq'_j
        > 
        \pi/2 - 5\delta - 2\tilde{\varepsilon},
    \end{equation}
    for $i=j+1,\ldots, k$.
    On the other hand, for the newly created strainer pairs $(p'_i, q'_i)$ and $(p'_j, q'_j)$ at the base point $z_j$, with $i=1,\ldots, j-1$, by assumption (c) above, we can apply the $(\varepsilon', \delta')$-weak quadruple condition to both quadruples $(z_j;p'_i,q'_i, p'_j)$ and $(z_j; p'_i, q'_i, q'_j)$ for $i=1,\ldots,j-1$.
    Therefore, the same argument as in Step~2--3 of the proof of Proposition~\ref{prop:bi_Lipschitz}, together with the $(\varepsilon',\delta')$-weak quadruple condition, implies that
    \begin{equation}\label{eq:almost_self_improvement_2_2}
        \tilde{\angle}p'_iz_j p'_j,\: \tilde{\angle}p'_iz_j q'_j,\: \tilde{\angle}q'_iz_jp'_j,\: \tilde{\angle}q'_iz_jq'_j
        > \pi/2 - 5\delta' - 2\varepsilon', 
    \end{equation}
    for $i=1,\ldots, j-1$.
    Finally, by our choice of $r_j$ and $z_j$, we have that $p'_j$ is a $(1,\delta')$-strainer at $z_j$ with the opposite strainer $q'_j$, and that $|xz_j|<\sum_{i=1}^{j}r_0/2^{i-1}$.
    These estimates, together with the estimates \eqref{eq:almost_self_improvement_2_1}, \eqref{eq:almost_self_improvement_2_2} and the choice of $r_j$ and $z_j$, imply that the new $k$-tuple $\pmb{p}_j:=(p_{j+1},\ldots,p_k,p'_1,\ldots, p'_j)$ is a $(k,\delta,\max\{\alpha, 5\delta + 2\tilde{\varepsilon}\})$-strainer at $z_j$ with the opposite strainer $\pmb{q}_j:=(q_{j+1},\ldots,q_k,q'_1,\ldots, q'_j)$.
    Moreover, its first $k-j$ entries form a $(k-j,\delta,\alpha)$-strainer at $z_j$ with the opposite strainer $(q_{j+1},\ldots,q_k)$, while its last $j$ entries form a $(j, \delta', 5\delta' + 2\varepsilon')$-strainer at $z_j$ with the opposite strainer $(q'_1,\ldots,q'_j)$. 
    
    \item Finally, after $k$ steps, we obtain a $k$-tuple $(p'_1,\ldots, p'_k)$ and a point $z_k$.
    By induction, it follows that this $k$-tuple is a $(k,\delta',5\delta' + 2\varepsilon')$-strainer at $z_k$, and that $|xz_k|<2r_0$.
    Since $r_0$ can be chosen so that $z_k\in U_1\subset U$, this proves the quantitative assertion of the lemma.
    The final assertion follows by choosing $\delta'>0$ and $\varepsilon'>0$ sufficiently small so that $0<\varepsilon'<\min\{\tilde{\varepsilon}, \rho\}$, $0<\delta'<\min\{\delta, \bar{\delta}(\varepsilon', S), \rho\}$, and $5\delta'+2\varepsilon'\leq \rho$.
  \end{enumerate}
\end{proof}

\section{Geometric measure-theoretic structure}\label{sect:geometric_measure_structure}
In this section, we establish several basic geometric measure-theoretic properties of $S$-concave Busemann concave spaces, concerning constancy of dimension, the measure contraction property, rectifiability, and almost everywhere uniqueness of Banach tangent cones.
Throughout this section and the next one, we assume that $X$ is not a singleton; the singleton case is trivial.

\medskip

\subsection{Dimension and measure contraction property}\label{sect:dim}
\noindent
We first introduce a slight variant of the strainer number, adapted to the notion of strainer used in the present paper; for the corresponding notion in Alexandrov spaces with curvature bounded below, see \cite{burago1992ad} and \cite[Definition~10.8.11]{burago2001course}.

\begin{definition}[Strainer number]
  Let $X$ be an $S$-concave Busemann concave space with $S\geq 1$.
  A natural number $m\in \mathbb{N}$ is called a strainer number at $x\in X$ if for any $0<\delta<1/2$ and any open neighborhood $U$ of $x$, there exists a $(m, \delta, \delta)$-strained point in $U$, while the analogous property fails for $m+1$.
  If no such $m$ exists, the strainer number at $x$ is defined to be $\infty$.
  The strainer number of $X$ is defined as the supremum of the strainer numbers over all points of $X$.
\end{definition}
It is clear that the strainer number is well-defined. 
Moreover, the following lemma shows that the strainer number at every point is at least $1$.

\begin{lemma}\label{lemma:strainer_number_at_least_1}
    Let $X$ be an $S$-concave Busemann concave space with $S\geq 1$.
    Then the strainer number at each point of $X$ is at least $1$.
\end{lemma}
\begin{proof}
    Fix $x\in X$, and choose $y\neq x$.
    Given any $0<\delta<1/2$, take $z$ on a unit-speed geodesic from $x$ to $y$ sufficiently close to $x$ so that
    \begin{equation}
      \max\{|xz|/|yz|, \delta_S(|xz|;|yz|)\} < \delta.
    \end{equation}
    It is then straightforward to check that $z$ is a $(1,\delta)$-strained point.
    Since $0<\delta<1/2$ is arbitrary, the strainer number at $x$ is at least $1$ by definition.
\end{proof}

\begin{lemma}\label{lemma:strainer_number_dimension}
  Let $X$ be an $S$-concave Busemann concave space with $S\geq 1$.
  Then the Hausdorff dimension of $X$ equals the strainer number at any point in $X$.
  If this number is finite, then it also coincides with the topological dimension of $X$.
\end{lemma}
\begin{proof}
  Let $U\subset X$ be an arbitrary bounded open neighborhood of $x\in X$.
  We show that the Hausdorff dimension of $U$ is equal to the strainer number $m$ at $x$.

  \emph{Case $m=\infty$.} 
  Given any $k\in \mathbb{N}$, let $0<\varepsilon<\delta_k$ be a positive number.
  By Proposition~\ref{lemma:verification_WQC_II}, we can find $\tilde{\delta}:=\tilde{\delta}(\varepsilon,S)>0$ such that $X$ satisfies the $(\varepsilon, \delta)$-weak quadruple condition for all $\delta\in (0,\tilde{\delta}]$.
  Let $0<\delta<\tilde{\delta}$ be a small number such that $3\delta + \varepsilon\leq \delta_k$.
  By the definition of strainer number, there exists a $(k,\delta,\delta)$-strained point $y\in U$.
  Since $2\delta + \delta +\varepsilon\leq 3\delta + \varepsilon\leq \delta_k$, by Corollary~\ref{cor:epsilon_openness_strainer_maps} it follows that there exists an open neighborhood $V\subset U$ of $y$ such that $f:V\to \mathbb{R}^k$ is $\sqrt{k}$-Lipschitz and $\bar{\varepsilon}_k$-open.
  Therefore, it holds that
  \begin{equation}
    k\leq \dim_H\left(f(V)\right)\leq \dim_H(V)\leq \dim_H(U)\leq \dim_H(X).
  \end{equation}
  Since $k\in \mathbb{N}$ is arbitrary, we obtain that $\dim_H(X)=\infty$.

  \emph{Case $m<\infty$.} 
  Let $0<\varepsilon<\delta_{m+1}/5$, and $\tilde{\delta}:=\tilde{\delta}(\varepsilon, S)$ be the constant given by Lemma \ref{lemma:verification_WQC_II} such that $X$ satisfies the $(\varepsilon, \delta)$-weak quadruple condition for all $\delta\in (0,\tilde{\delta}]$.
  Let $0<\delta<\tilde{\delta}$ be an arbitrarily small number such that $12\delta + 5\varepsilon \leq \delta_{m+1}$.
  
  By the definition of strainer number, there exists an open neighborhood $V\subset U$ of $x$ and $\bar{\delta}>0$ such that $V$ does not contain any $(m+1,\bar{\delta}, \bar{\delta})$-strained point.
  By the self-improvement property (Lemma~\ref{lemma:almost_self_improvement}), we conclude that $V$ also does not contain any $(m+1,\delta,5\delta + 2\varepsilon)$-strained point.
  Indeed, if the claim were not true and $V$ contained a $(m+1,\delta,5\delta + 2\varepsilon)$-strained point, then our choice of $\delta,\varepsilon$ would imply
  \begin{equation}
    7\delta + \left(5\delta + 2\varepsilon\right) + 3\varepsilon
    \leq
    12\delta + 5\varepsilon
    \leq
    \delta_{m+1}.
  \end{equation}
  By Lemma~\ref{lemma:almost_self_improvement}, it would follow that $V$ contains a $(m+1, \delta',\delta')$-strained point for any $0<\delta'<\bar{\delta}$, which contradicts our assumption on $V$.
  On the other hand, since the strainer number at $x$ equals $m$, the definition of strainer number implies that $V$ contains a $(m,\delta,\delta)$-strained point $z$.
  Since $2\delta +\delta +\varepsilon\leq \delta_{m+1}<\delta_m$ and $V$ admits no $(m+1,\delta, \max\{5\delta+2\varepsilon,\delta\})$-strained point, the hypotheses of Proposition~\ref{prop:bi_Lipschitz} are satisfied.
  Therefore, there exist an open neighborhood $W\subset V$ of $z$ and a map $f:W\to \mathbb{R}^m$ such that $W$ is bi-Lipschitz homeomorphic to an open domain $f(W)$ of $\mathbb{R}^m$.
  Since all bounded open subsets of a Busemann concave space have the same Hausdorff dimension (see \cite[Lemma 2.22]{kell2019sectional}), and since $X$ can be covered by countably many bounded open subsets, we obtain
  \begin{equation}
    m=\dim_{H}f(W)= \dim_H(W)= \dim_H(U) = \dim_H(X).
  \end{equation}
  In particular, the strainer number at every point of $X$ coincides with the Hausdorff dimension of $X$.
  This proves the first assertion.
  The second assertion follows by exactly the same argument as in \cite[Lemma 6.3]{han2025structure}, where we only use the bi-Lipschitz homeomorphism from the first assertion together with the relationship between topological dimension and Hausdorff dimension; see \eqref{eq:relation_different_dims} in Section~\ref{subsect:basics}.
\end{proof}

We now state the main result of this subsection.
In what follows, we say that a metric space has \emph{constant Hausdorff dimension} if the local Hausdorff dimensions at all points coincide.
Here the local Hausdorff dimension $\dim_H X(x)$ at a point $x$ is defined as
\begin{equation}
  \dim_{H}X(x):=\inf\left\{\dim_H(U): x\in U,\: U\text{ open}\right\}.
\end{equation}
\begin{proposition}\label{prop:MCP}
  Let $X$ be an $S$-concave Busemann concave space with $S\geq 1$.
  Then $X$ has finite Hausdorff dimension if and only if it has finite strainer number.
  In this case, these two quantities coincide with the topological dimension of $X$.
  Denoting this common integer by $n\geq 1$, the space $X$ has constant Hausdorff dimension $n$ and carries a nontrivial $n$-dimensional Hausdorff measure.
  Moreover, $(X,\mathsf d,\mathcal{H}^n)$ satisfies the measure contraction property $\mathrm{MCP}(0,n)$ and the Bishop--Gromov volume comparison inequality $\mathrm{BG}(0,n)$.
  In particular, $(X,\mathsf d)$ is doubling and proper.
\end{proposition}
\begin{proof}
  The first assertion follows directly from Lemmas~\ref{lemma:strainer_number_dimension} and~\ref{lemma:strainer_number_at_least_1}.
  The constancy of Hausdorff dimension follows directly from \cite[Lemma 2.22]{kell2019sectional}.
  The existence of nontrivial $n$-dimensional Hausdorff measure follows from Lemma~\ref{lemma:strainer_number_dimension}: there exists a bounded open subset $U$ of $X$ that is bi-Lipschitz homeomorphic to a bounded open subset of $\mathbb{R}^n$, and hence $\mathcal{H}^n(U)\in (0,\infty)$.
  The remaining assertions follow by exactly the same argument as in \cite[Theorem 6.4]{han2025structure}.
  They are direct consequences of the existence of nontrivial $n$-dimensional Hausdorff measure and the geometric properties of $\mathrm{MCP}$ spaces. 
  No local semi-convexity is used in this step.
\end{proof}

In view of the preceding proposition, we say that an $S$-concave Busemann concave space is \emph{finite dimensional} if its Hausdorff dimension, or equivalently, if its strainer number is finite.


\subsection{Rectifiability and Banach tangent cones}\label{subsect:rectifiability_banach_tangent_cones}

\noindent
In this subsection, we show that every $n$-dimensional $S$-concave Busemann concave space is $n$-rectifiable. We also show that $\mathcal{H}^n$-a.e. point admits a unique tangent cone isometric to an $n$-dimensional Banach space.
These properties are well-known for finite-dimensional Alexandrov spaces with curvature bounded below (see, for example, \cite[Chapter 10]{burago1992ad}).
They were also established for locally semi-convex, $S$-concave Busemann concave spaces in \cite{han2025structure}, using the local bi-Lipschitz property of strainer maps and Kirchheim's local structure theorem for rectifiable sets in metric spaces \cite[Theorem 9]{kirchheim1994rectifiable}.
Here, we obtain these properties directly from a general rectifiability criterion for non-collapsed $\mathrm{MCP}$ spaces with unique tangent cones, due to Magnabosco--Mondino--Rossi \cite{mondino2025rectifiability}.

We state a version of this result needed below.
Before stating it, recall that, for a metric space $(X,\mathsf d)$ and $\alpha\in \mathbb{R}^+$, the lower and upper densities of the $\alpha$-dimensional Hausdorff measure $\mathcal{H}^{\alpha}$ at a point $x\in X$ are defined as:
\begin{equation}
  \Theta_{*\alpha}(x):=\liminf_{r\to 0}\frac{\mathcal{H}^{\alpha}\left(B(x,r)\right)}{\omega_nr^{\alpha}},\quad
  \Theta^*_{\alpha}(x):=\limsup_{r\to 0}\frac{\mathcal{H}^{\alpha}\left(B(x,r)\right)}{\omega_nr^{\alpha}}.
\end{equation}

\begin{lemma}[\cite{mondino2025rectifiability}, Theorem 1.2]\label{thm:MMR_thm}
    Let $(X,\mathsf d,\mathcal{H}^N)$ be an $\mathrm{MCP}(K,N)$ space, where $K\in \mathbb{R}$ and $N\in (1,\infty)$ (or $K\leq 0$ and $N=1$).
    Assume that $\mathcal{H}^N$-a.e.\ point of $X$ admits a unique Gromov--Hausdorff tangent cone, and that, for $\mathcal{H}^N$-a.e.\ point $x\in X$, the lower and upper density of $\mathcal{H}^N$ satisfy $0<\Theta_{*N}(x)\leq\Theta^*_N(x)<\infty$.
    Then $N$ is an integer, $(X,\mathsf d)$ is $N$-rectifiable, and for $\mathcal{H}^N$-a.e.\ point $x$, the family $\mathrm{Tan}(X,\mathsf d, x)$ consists of a single element (up to an isometry) which is isometric to an $N$-dimensional Banach space.
\end{lemma}
\begin{remark}
    Although \cite[Theorem 1.2]{mondino2025rectifiability} is stated for $K\in\mathbb{R}$ and $N\in(1,\infty)$, its conclusion remains valid for $K\leq 0$ and $N=1$. 
    This is essentially because \cite[Theorem 4.3]{mondino2025rectifiability} holds trivially for $N=1$, which further implies that $\mathcal{H}^1$-a.e.\ point of $X$ admits a unique tangent cone isometric to a commutative sub-Finsler Carnot group with homogeneous dimension $N=1$; see the proof of \cite[Theorem~4.6]{mondino2025rectifiability}, together with \cite[Theorem~3.18]{mondino2025rectifiability}, and also \cite[Remark 5.15]{fujioka2026busemann}.
\end{remark}

We are now in a position to state the main result of this subsection.
\begin{proposition}\label{prop:rectifiability_Banach_tangent_cone}
  Let $X$ be an $n$-dimensional $S$-concave Busemann concave space with $S\geq 1$.
  Then $X$ is $n$-rectifiable, and $\mathcal{H}^n$-a.e.\ point of $X$ admits a unique tangent cone which is isometric to an $n$-dimensional Banach space.
\end{proposition}
\begin{proof}
  By Proposition~\ref{prop:MCP}, $(X,\mathsf d,\mathcal{H}^n)$ is a doubling Busemann concave space satisfying $\mathrm{MCP}(0,n)$.
  Hence, every point $x\in X$ admits a unique Gromov--Hausdorff tangent cone which coincides with $T_xX$ up to an isometry; see Section~\ref{subsect:tangent_cones} or \cite[Corollary 2.21]{kell2019sectional}.
  It remains only to establish the lower and upper density bounds at $\mathcal{H}^n$-a.e.\ point of $X$.

  For the lower bound, Bishop--Gromov volume comparison for $\mathrm{MCP}(0,n)$ spaces implies that the lower and upper densities of $\mathcal{H}^n$ at every point coincide and satisfy
  \begin{equation}
        \Theta^*_n(x)=\Theta_{*n}(x)
        =
        \Theta_n(x):=\lim_{r\to 0}\frac{\mathcal{H}^n(B(x,r))}{\omega_nr^n}
        \geq
        \frac{\mathcal{H}^n(B(x,r_x))}{\omega_n (r_x)^n}>0,
  \end{equation}
  for every $x\in X$, where $r_x>0$ can be chosen arbitrarily.
  This is because, for any $r_x>0$, the proof of Lemma \ref{lemma:strainer_number_dimension} implies that there exists a bounded nonempty open subset $W\subset B(x,r_x)$ which is bi-Lipschitz homeomorphic to a bounded nonempty domain of $\mathbb{R}^n$. 
  This implies that $\mathcal{H}^n(B(x,r_x))\geq \mathcal{H}^n(W)>0$.
  In particular, every nonempty open ball has positive $\mathcal{H}^n$-measure.
  For the upper bound, set $E:=\{x\in X: \Theta_n(x)=\infty\}$.
  By a standard result in geometric measure theory (see, for example, \cite[eq. (2.1)]{ambrosio2000rectifiable}), for every $1<t<\infty$, we have the following implication:
  \begin{equation}
    \Theta^*_n(x)\geq t \quad \text{for all } x\in E
    \quad \Longrightarrow \quad
    \mathcal{H}^n \geq t\,\mathcal{H}^n\llcorner E.
  \end{equation}
  This forces $\mathcal{H}^n(E)=0$, and hence $\Theta_n(x)<\infty$ for $\mathcal{H}^n$-a.e.\ point $x\in X$.
  Therefore, by Lemma~\ref{thm:MMR_thm}, the space $X$ is $n$-rectifiable, and $\mathcal{H}^n$-a.e.\ point admits a unique tangent cone isometric to an $n$-dimensional Banach space.
\end{proof}

\begin{remark}
  Rectifiability and the a.e.\ uniqueness of Banach tangent cones for finite-dimensional $S$-concave Busemann concave spaces can also be proved by more direct arguments using strainer maps, as in \cite[Section 7.2]{han2025structure}. 
  In the next section, we use strainer maps to establish finer geometric structure results for finite-dimensional $S$-concave Busemann concave spaces.
\end{remark}

We conclude this subsection with a further geometric characterization of tangent cones of finite-dimensional $S$-concave Busemann concave spaces, under the additional assumption of local $p$-uniform convexity.
\begin{corollary}\label{cor:characterize_Banach_tangent_cone}
    Let $X$ be an $n$-dimensional $S$-concave Busemann concave space.
    Suppose that for every point $x \in X$, there exists $r_x > 0$ such that the following $p$-uniform convexity condition holds on $B(x,r_x)$. 
    For every constant-speed geodesic $\xi: [0,1] \to B(x, r_x)$ and every point $y \in B(x, r_x)$, one has
    \begin{equation}\label{eq:characterize_Banach_tangent_cone}
        |y\xi(1/2)|^p
        \leq
        \frac{1}{2}|y\xi(0)|^p + \frac{1}{2}|y\xi(1)|^p - \frac{C}{4}\left|\xi(0)\xi(1)\right|^p,
    \end{equation}
    for some $p \geq S+1$ and $C = 4/2^p$.
    Then all tangent cones of $X$ are uniquely geodesic.
    Moreover, at $\mathcal{H}^n$-almost every point $x$, the unique tangent cone $(T_x, \mathsf{d}_x, o)$ is isometric to a $2$-uniformly smooth, $p$-uniformly convex Banach space $(E_x,\|\cdot\|_x, o_x)$ with strictly convex norm $\|\cdot\|_x$.
\end{corollary}
\begin{proof}
  The proof is essentially the same as that of \cite[Corollary 7.9]{han2025structure}.
  Indeed, the argument there for the uniquely geodesic property relies only on the local $p$-uniform convexity assumption \eqref{eq:characterize_Banach_tangent_cone} and does not use local semi-convexity.
  Moreover, the same proof remains valid after replacing the norm $\|\cdot\|_x$ of the Banach tangent cone by the distance $\mathsf d_x$ of the tangent cone $T_xX$.
  The second assertion then follows by the same argument as in the last part of the proof of \cite[Corollary 7.9]{han2025structure}.
\end{proof}

\section{Manifold and singular structures}\label{sect:manifold_singular_structures}
In this section, we further investigate the geometric structure of finite-dimensional $S$-concave Busemann concave spaces.

\medskip

\subsection{Topological manifold part}\label{subsect:topological_manifold_part}

\noindent
In this subsection, we show that each finite-dimensional $S$-concave Busemann concave space contains an open dense topological manifold part of top dimension and full Hausdorff measure.
More precisely, we show that, for suitable choices of
$\delta,\alpha>0$, the $n$-almost regular set $\mathcal{A}(n,\delta,\alpha)$ provides the desired topological manifold part.

We first introduce the notions of almost regular sets and regular points.
\begin{definition}[Almost regular sets]
  Let $X$ be an $S$-concave Busemann concave space with $S\geq 1$.
  Given $k\in \mathbb{N}$ and $0<\delta,\alpha<1/2$, we define $\mathcal{A}(k,\delta,\alpha)\subset X$, called the \emph{$k$-almost regular set} of $X$, to be the set of all $(k,\delta,\alpha)$-strained points of $X$\footnote{When $k=1$, $\mathcal{A}(1,\delta,\alpha)$ is the set of all $(1,\delta)$-strained points; the parameter $\alpha$ plays no role in this case.}.
  We say that a point $x\in X$ is a \emph{regular point} if its tangent cone $T_xX$ is isometric to a Banach space.
\end{definition}
\begin{remark}\label{rmk:full_measure_1_almost_regular_set}
  For every $n$-dimensional $S$-concave Busemann concave space, the $1$-almost regular set $\mathcal{A}(1,\delta,\delta)$ has full $\mathcal{H}^n$-measure for every $0<\delta<1/2$. 
  This follows from the almost extendability of geodesics in $\mathrm{MCP}$ spaces and does not require local semi-convexity; see \cite[Lemma 7.2]{han2025structure} and \cite[Section 3.3]{von2008local}.
\end{remark}

We next prove a key technical lemma, estimating the size of the set of points which admit lower-dimensional strainers but no higher-dimensional strainers.
Analogous estimates were obtained for Alexandrov spaces with curvature bounded below in \cite[Lemma 10.5 and Theorem 10.6]{burago1992ad}, using the metric cone structure of tangent cones and cardinality estimates for maximal separated subsets, and for locally semi-convex, $S$-concave Busemann concave spaces in \cite[Lemma 7.3]{han2025structure}, using the almost extendability of geodesics and the infinitesimal behavior of strainer maps restricted to singular sets.

Here we follow an approach similar to that of \cite[Lemma 7.3]{han2025structure}.
This method allows us to estimate the top-dimensional Hausdorff measure of singular sets without appealing to the structure of tangent cones.
Most of the argument is unchanged, except in Step 2 of the proof, where we use the weak quadruple condition in place of both the continuity of angles viewed from a point and the identity that the sum of adjacent angles viewed from a common point equals $\pi$. 
In \cite{han2025structure}, these local geometric properties rely on the local semi-convexity assumption, which is not assumed in the present setting.

Recall that, for a point $x\in X$, the set $I_x\subset X$ consists of all non-cut points of $x$, that is, all points $y\in X$ such that some unit-speed geodesic from $x$ to $y$ can be extended beyond $y$ locally.

\begin{lemma}\label{lemma:full_measure}
  Let $X$ be an $n$-dimensional $S$-concave Busemann concave space with $S\geq 1$.
  Fix $1\leq k\leq n-1$, and let $\tilde{\varepsilon}\in [0,\delta_k)$ and $\tilde{\delta}\in (0,\pi]$ be such that $X$ satisfies the $(\tilde{\varepsilon}, \tilde{\delta})$-weak quadruple condition.
  Let $(p_1,\ldots, p_k)$ be a $(k,\delta,\alpha)$-strainer on an open subset $U\subset X$ with common opposite strainer $(q_1,\ldots,q_k)$, where $\delta,\alpha>0$ satisfy 
  \begin{equation}
    \delta\leq \tilde{\delta}
    \quad 
    \text{and}
    \quad 
    2\delta + \alpha + \tilde{\varepsilon} \leq \delta_k.
  \end{equation}
  Let $E\subset U$ be the set of points which admit no $(k+1,\delta, \max\{\alpha, 5\delta + 2\tilde{\varepsilon}\})$-strainer. 
  Then $\mathcal{H}^n(E)=0$.
\end{lemma}
\begin{proof}
  Let $f:=(\mathsf d_{p_1},\ldots,\mathsf d_{p_k}):U\to \mathbb{R}^k$ be the associated strainer map.
  Suppose, for contradiction, that $\mathcal{H}^n(E)>0$.

  \begin{enumerate}[fullwidth, label=\textbf{Step \arabic*:}]
    \item By Proposition~\ref{prop:MCP}, the $n$-dimensional Hausdorff measure $\mathcal{H}^n$ is a nontrivial doubling Radon measure on $X$. 
    Moreover, for each $x\in X$, the set $I_x$ of non-cut points of $x$ has full $\mathcal{H}^n$-measure, since the set of cut points of a point in an $\mathrm{MCP}$ space is negligible; see, for example, \cite[Section 3.3]{von2008local}.
    Therefore, following exactly the same argument in Step 1 of \cite[Lemma 7.3]{han2025structure}, we can find a subset $K\subset E$ such that $\mathcal{H}^n(E\setminus K)=0$, $I_x\cap B(y,r)\cap K\neq \emptyset$ for all $x,y\in K$ and $r>0$, and $I_x\cap K$ is dense in $K$ for every $x\in K$.
    We emphasize that Step 1 of \cite[Lemma 7.3]{han2025structure} relies only on the Lebesgue differentiation theorem for doubling measures, and does not use local semi-convexity.

    \item In this step, we show that
    \begin{equation}\label{eq:full_measure_0}
      \liminf_{K\ni y\to x}\frac{|f(x)f(y)|}{|xy|}\geq \delta/8,\quad \text{for all } x\in K.
    \end{equation}
    Since $I_x\cap K$ is dense in $K$ for every $x\in K$, it suffices to prove that
    \begin{equation}\label{eq:full_measure_1}
      \lim_{r\to 0}\inf_{y\in B(x,r)\cap K\cap I_x}\frac{|f(x)f(y)|}{|xy|}\geq \delta/8,\quad \text{for all } x\in K.
    \end{equation}
    The non-emptiness of $B(x,r)\cap K\cap I_x$ for every $x\in K$ and $r>0$ ensures that the infimum on the left-hand side of \eqref{eq:full_measure_1} is not infinite\footnote{We use the convention that $\inf \emptyset =\infty$.} and is bounded above by the Lipschitz constant of $f$.

    Suppose by contradiction that there exists $x\in K$ and a sequence $\{x_j\}_j\subset K\cap I_x$ with $x_j\to x$ such that
    \begin{equation}\label{eq:full_measure_2}
        \lim_{j\to \infty}\frac{\big||p_ix_j|-|p_ix|\big|}{|x_jx|}
        \leq
        \lim_{j\to \infty}\frac{\left|f(x)f(x_j)\right|}{|x_jx|}
        < \delta/8,\quad \text{for all } i=1,\ldots, k.
    \end{equation}
    Set $t_j:=|xx_j|$.
    Choose $j$ sufficiently large such that 
    \begin{equation}\label{eq:full_measure_3}
      \max\{4t_j/|q_kx|, \delta_S(4t_j;|q_kx|)\}<0.01\delta.
    \end{equation}
    Since $x_j\in I_x$, we may choose a point $y$ on an extension of a unit-speed geodesic from $x$ to $x_j$ beyond $x_j$ such that
    \begin{equation}\label{eq:full_measure_4}
      \max\{|x_jy|/t_j, \delta_S(|x_jy|;t_j)\}<0.01\delta.
    \end{equation}
    
    We claim that $(p_1,\ldots,p_k,x)$ is a $(k+1,\delta, \max\{\alpha, 5\delta + 2\tilde{\varepsilon}\})$-strainer at $x_j$.
    Indeed, by the choice of $y$ and $t_j$, namely \eqref{eq:full_measure_4}, it follows that $x$ is a $(1,\delta)$-strainer at $x_j$ with the opposite strainer $y$.
    Moreover, \eqref{eq:full_measure_3} gives $|q_kx|>|xx_j|=t_j$ and $\delta_S(|xx_j|;|q_kx|)<\delta$.
    Thus, conditions (1)--(2) in the definition of strainers (Definition~\ref{def:k_strainer}) are satisfied for $(p_1,\ldots,p_k,x)$ at $x_j$.
    It remains only to verify the almost orthogonality condition (3) in Definition~\ref{def:k_strainer}.

    Fix $i\in \{1,\ldots,k\}$.
    Consider the comparison triangle $\tilde{\Delta}p_ix_jx$.
    By the choice of $t_j$ and $x_j$, namely \eqref{eq:full_measure_3} and \eqref{eq:full_measure_4}, together with the hypothesis \eqref{eq:full_measure_2}, the same argument in (7.5)--(7.6) of Step 2 in \cite[Lemma 7.3]{han2025structure} (see also \cite[Lemma 4.3, eq. (4.8)]{han2025structure}) gives 
    \begin{equation}
        \left|\cos \tilde{\angle}p_ix_j x\right|
        \leq
        \left|\frac{|p_ix_j|-|p_ix|}{t_j}\right| + \frac{t_j}{2|p_ix|}
        <
        \delta/8 + 0.01\delta < \delta/4.
    \end{equation}
    Note that this argument uses only the Euclidean law of cosines and $S$-concavity.
    Hence we obtain $|\tilde{\angle}p_ix_jx-\pi/2|<\delta/2$.

    It remains to check the almost orthogonality of the other three comparison angles
    \[
      \tilde{\angle}q_ix_jx,\quad
      \tilde{\angle}p_ix_jy,\quad
      \tilde{\angle}q_ix_jy.
    \]
    By the choice of $t_j$ and $y$, we have
    \begin{equation}
      |yx_j|<|xx_j|\leq 0.01\delta\min\{|p_ix_j|,|q_ix_j|\}.
    \end{equation}
    Applying the same argument as in Step 2 of Proposition~\ref{prop:bi_Lipschitz} to the quadruples $(x_j; p_i,q_i,x)$ and $(x_j; p_i,q_i,y)$\footnote{Here $(x,y)$ is a $(1,\delta)$-strainer pair, whereas in Proposition~\ref{prop:bi_Lipschitz} the $(1,\delta)$-strainer pair is $(y,x)$.}, we get
    \begin{equation}
      \left|\tilde{\angle}p_ix_jx + \tilde{\angle}q_ix_jx - \pi\right|< 2\delta + \tilde{\varepsilon},
      \quad
      \left|\tilde{\angle}p_ix_jy + \tilde{\angle}q_ix_jy - \pi\right|< 2\delta + \tilde{\varepsilon}.
    \end{equation}
    Together with the inequality $|\tilde{\angle}p_ix_jx-\pi/2|<\delta/2$ obtained above, the same argument as in Step 3 of Proposition~\ref{prop:bi_Lipschitz} implies that $\tilde{\angle}q_ix_jx, \tilde{\angle}p_ix_jy, \tilde{\angle}q_ix_jy$ are all greater than $\pi/2- 5\delta - 2\tilde{\varepsilon}$.

    Thus, $(p_1,\ldots, p_k, x)$ is a $(k+1,\delta, \max\{\alpha, 5\delta + 2\tilde{\varepsilon}\})$-strainer at $x_j\in E$, contradicting the definition of $E$.
    This proves the claim \eqref{eq:full_measure_0}.

    \item Finally, following exactly the same argument as in Step 3 of \cite[Lemma 7.3]{han2025structure}, we obtain $\mathcal{H}^n(K)=0$.
    That argument uses only the infinitesimal behavior \eqref{eq:full_measure_0} of $f$ restricted to $K$, based on \cite[Lemma 3.1]{lytchak2006open}, and does not rely on local semi-convexity.
    This contradicts our initial assumption together with the choice of $K$, namely, $\mathcal{H}^n(K)=\mathcal{H}^n(E)>0$.
    Hence the initial assumption is false, and the assertion follows.
  \end{enumerate}
\end{proof}

We are now in a position to prove the main result of this subsection: for a suitable choice of parameters, the almost regular set is an open dense topological manifold of top dimension and full top-dimensional Hausdorff measure.

\begin{theorem}[Topological manifold part]\label{thm:top_manifold_part}
  Let $X$ be an $n$-dimensional $S$-concave Busemann concave space with $S\geq 1$.
  Then, for every $0<\delta<1/2$, the $n$-almost regular set $\mathcal{A}(n,\delta, \delta)$ is open and dense in $X$, and has full $\mathcal{H}^n$-measure; that is, $\mathcal{H}^n(X\setminus \mathcal{A}(n,\delta,\delta))=0$.
  Furthermore, if $0<\delta<1/2$ is sufficiently small, then $\mathcal{A}(n,\delta,\delta)$ is a topological $n$-manifold.
  In particular, $X$ contains an open dense topological manifold part of top dimension and full top-dimensional Hausdorff measure.
\end{theorem}
\begin{proof}
  \begin{enumerate}[fullwidth]
    \item[\textit{Openness and density:}] by Proposition~\ref{prop:almost_orthogonality}, $\mathcal{A}(n,\delta,\alpha)$ is open for all $0<\delta,\alpha<1/2$.
    To prove density, fix $\delta, \alpha>0$ and choose $0<\delta'<\min\{\delta, \alpha\}$. 
    Since the strainer number at every point equals the Hausdorff dimension $n$ of $X$, by the definition of strainer number, every neighborhood of every point contains a $(n,\delta',\delta')$-strained point.
    Such a strained point belongs to $\mathcal{A}(n,\delta, \alpha)$, and hence $\mathcal{A}(n,\delta,\alpha)$ is dense in $X$.

    \item[\textit{Topological manifold part:}] we show that $\mathcal{A}(n,\delta,\delta)$ is a topological $n$-manifold for all sufficiently small $0<\delta<1/2$.
    Let $0<\varepsilon<\delta_{n+1}/5$ and let $\tilde{\delta}:=\tilde{\delta}(\varepsilon,S)$ be the constant given by Lemma~\ref{lemma:verification_WQC_II}, such that $X$ satisfies the $(\varepsilon,\delta)$-weak quadruple condition for all $\delta\in (0,\tilde{\delta}]$.
    Fix any $0<\delta<\tilde{\delta}$ such that $12\delta+ 5\varepsilon\leq \delta_{n+1}$.

    Let $x\in \mathcal{A}(n,\delta,\delta)$ be arbitrary.
    By the openness of almost regular sets, the definition of strainer number and the self-improvement property, as in the proof of Lemma~\ref{lemma:strainer_number_dimension}, there exists a small open neighborhood $U\subset \mathcal{A}(n,\delta,\delta)$ of $x$ such that $U$ contains no $(n+1, \delta, 5\delta+2\varepsilon)$-strained point.
    Since our choices of $\varepsilon$ and $\delta$ imply that $2\delta + \delta+ \varepsilon \leq \delta_{n}$, Proposition~\ref{prop:bi_Lipschitz} gives a possibly smaller open neighborhood $V\subset U$ of $x$ which is bi-Lipschitz homeomorphic to an open subset of $\mathbb{R}^n$.
    Hence, $\mathcal{A}(n,\delta,\delta)$ is a topological $n$-manifold.

    \item[\textit{Full Hausdorff measure:}] we show that, for every $0<\delta<1/2$, the $n$-almost regular set $\mathcal{A}(n,\delta,\delta)$ has full $\mathcal{H}^n$-measure.
    When $n=1$, this follows directly from the almost extendability of geodesics in $\mathrm{MCP}$ spaces; see Remark~\ref{rmk:full_measure_1_almost_regular_set}.
    Therefore, we may assume that $n\geq 2$.

    Let $0<\varepsilon<\delta_{n}/3$, and $\tilde{\delta}:=\tilde{\delta}(\varepsilon,S)$ be the corresponding constant given by Lemma~\ref{lemma:verification_WQC_II} such that $X$ satisfies the $(\varepsilon,\delta)$-weak quadruple condition for all $\delta\in (0,\tilde{\delta}]$.
    Choose $0<\delta<\tilde{\delta}$ sufficiently small such that $7\delta + 3\varepsilon \leq \delta_{n}$.

    We claim that, for each $k=2,\ldots, n$, the set 
    \begin{equation}
      \mathcal{A}(k-1,\delta, 5\delta+ 2\varepsilon)\setminus \mathcal{A}(k, \delta, 5\delta + 2\varepsilon)
    \end{equation}
    has zero $\mathcal{H}^n$-measure.
    Indeed, for any $x\in \mathcal{A}(k-1,\delta, 5\delta+ 2\varepsilon)$, Proposition~\ref{prop:almost_orthogonality} gives a sufficiently small radius $r_x>0$ such that the ball $B(x,r_x)$ admits a $(k-1, \delta, 5\delta + 2\varepsilon)$-strainer with a common opposite strainer.
    Since $2\delta + (5\delta + 2\varepsilon)+\varepsilon\leq \delta_{k-1}$, Lemma~\ref{lemma:full_measure} implies that the set $B(x,r_x)\setminus \mathcal{A}(k,\delta,5\delta + 2\varepsilon)$ is $\mathcal{H}^n$-negligible.
    As $X$ is proper by Proposition~\ref{prop:MCP}, we may extract a countable cover $\{B(x_l,r_{x_l})\}_{l\in \mathbb{N}}$ of $\mathcal{A}(k-1,\delta, 5\delta + 2\varepsilon)$ by such balls.
    Hence, we conclude that $\mathcal{A}(k-1,\delta, 5\delta+ 2\varepsilon)\setminus \mathcal{A}(k, \delta, 5\delta + 2\varepsilon)$ is $\mathcal{H}^n$-negligible for every $k=2,\ldots, n$.
    By subadditivity of measure, it follows that
    \begin{multline}
      \mathcal{H}^n\left(X\setminus \mathcal{A}(n,\delta, 5\delta + 2\varepsilon)\right)
      \leq
      \mathcal{H}^n\left( X\setminus \mathcal{A}(1,\delta)\right)\\
      +
      \sum_{k=2}^n \mathcal{H}^n\left( \mathcal{A}\left(k-1, \delta, 5\delta + 2\varepsilon \right) \setminus \mathcal{A}\left(k,\delta, 5\delta + 2\varepsilon \right) \right)
      =0,
    \end{multline}
    where we use the identification $\mathcal{A}(1,\delta, \alpha)=\mathcal{A}(1,\delta)$ for all $\alpha>0$.
    Thus, $\mathcal{A}(n,\delta, 5\delta + 2\varepsilon)$ has full $\mathcal{H}^n$-measure.

    Finally, given any $0<\delta'<1/2$, choose $0<\varepsilon<\delta_n/3$ and $0<\delta<\tilde{\delta}(\varepsilon,S)$ sufficiently small such that $7\delta + 3\varepsilon\leq \delta_n$ and $5\delta + 2\varepsilon \leq \delta'$.
    Then $\mathcal{A}(n,\delta, 5\delta+2\varepsilon)\subset \mathcal{A}(n,\delta',\delta')$, and therefore $\mathcal{A}(n,\delta',\delta')$ has full $\mathcal{H}^n$-measure.
    This proves the claim.
  \end{enumerate}
\end{proof}

\subsection{Singular strata and stratification}\label{subsect:singular_strata}

\noindent
In this subsection, we estimate the Hausdorff dimension of singular sets.
This estimate yields a natural measure-theoretic stratification $\{X_k\}_{k=0}^n$ for every $S$-concave Busemann concave space.
For this stratification, $\dim_H(X_k)\leq k$ for each $k=0,\ldots,n$, and the top-dimensional stratum $X_n$ is an open dense topological $n$-manifold.
A related, stronger topological stratification theorem for finite-dimensional Alexandrov spaces with curvature bounded below was established in \cite{perel1994elements}; see also \cite{perelman1991alexandrov,burago2001course,perel1994extremal} and \cite[Section 13]{burago1992ad}.
For locally semi-convex, $S$-concave Busemann concave spaces, a measure-theoretic stratification was also obtained in \cite{han2025structure}.

We first estimate the Hausdorff dimension of singular sets. 
As observed in \cite{han2025structure}, the method of \cite{burago1992ad} for Alexandrov spaces with curvature bounded below implicitly relies on the metric cone structure of tangent cones. 
This approach therefore does not apply directly in our setting, since tangent cones need not be metric cones and the relevant notions of angle need not coincide.
Instead, we follow the approach of \cite{han2025structure}, which is based on the quasi-metric structure and uniform compactness of spaces of directions with common length, together with a technical estimate quantifying how side-length ratios control the discrepancy between angles viewed from a fixed point and angles of common scale, when the latter are small.

We need the following two technical lemmas.
The first one is taken from \cite[Lemma 7.11]{han2025structure}, and is adapted from \cite[Lemma 10.2]{lytchak2019geod} and \cite[Lemma 10.3]{burago1992ad}. 
It will be used to bound the maximal cardinality of separated subsets of singular sets.
\begin{lemma}\label{lemma:number_points}
    For every $N, L \geq 1$ and every natural number $M \geq 1$, there exists a constant $\bar{K} = \bar{K}(N, L, M)$ with the following property: if $X$ is an $N$-doubling metric space and $E \subset X$ is an $r$-separated subset for some $r>0$ containing at least $\bar{K}$ elements, then there exist $M$ distinct points $\{x_i\}_{i=0}^{M-1} \subset E$ such that $|x_{i+1} x_0| \geq L |x_i x_0|$ for each $i = 1, \ldots, M-2$.
\end{lemma}

The second lemma, taken from \cite[Lemma 7.12]{han2025structure}, provides quantitative control of the discrepancy between angles viewed from a point and angles of common scale, when the latter are small.
We note that \cite[Lemma 7.12]{han2025structure} was established for $S$-concave Busemann concave spaces without assuming local semi-convexity.
Recall that, for $\delta\in (0, 1)$, the beta number $\beta(\delta)$\footnote{This beta number is unrelated to the cardinality function defined in Section~\ref{subsect:basics}.} is defined as
\begin{equation}\label{def:beta_delta_num}
    \beta(\delta):=\frac{(1-\cos\delta)\sin\delta}{2(1+\sin\delta)}.
\end{equation}
Note that $\beta(\delta)< 1/4$ and $\beta(\delta)=O(\delta^3)$ as $\delta\to 0$.
\begin{lemma}\label{lemma:quantitive_discrepancy_angles}
    Let $X$ be an $S$-concave Busemann concave space, and let $0<\delta <1/2$.
    Then there exists a constant $L_0 = L_0(\delta, S) > 1$ such that the following holds: let $x,y,z\in X$ be points with $|xz|/|xy| \geq L_0$, and let $\gamma, \eta$ be unit-speed geodesics from $x$ to $y$ and to $z$, respectively. 
    If $\angle_x(\gamma(r), \eta(r)) < \beta(\delta)$ for some $r > 0$, then $\tilde{\angle}_x(y,z)< \delta$.
    Moreover, $z$ is a $(1, 2\delta)$-strainer at $y$ with opposite strainer $x$.
\end{lemma}
\begin{proof}
  Fix $0<\delta<1/2$.
  The choice of $L_0$ and the proof of the first assertion are exactly the same as in \cite[Lemma 7.12]{han2025structure}.
  For the second assertion, observe that $\delta_S$ is bounded above by the error function $\bar{\delta}_S(\tau;t):=\arccos(1-S\tau/(2t))$ used there.
  Hence, the same argument shows that our choice of $L_0$ implies 
  \begin{equation}
    \tilde{\angle}zyx>\pi-2\delta,\quad 
    \delta_{S}(|xy|;|yz|)< \bar{\delta}_S(|xy|;|yz|)<\delta,
    \quad\text{and}\quad
    |xy|<|yz|.
  \end{equation}
  By the definition of strainers, $z$ is therefore a $(1,2\delta)$-strainer at $y$ with opposite strainer $x$.
\end{proof}

We are now ready to state the key lemma of this subsection, which estimates the maximal cardinality of separated subsets of singular sets inside a cylindrical region.
Before doing so, we recall the following variant of the notion of \emph{$R$-long strainers}, originally introduced in \cite{burago1992ad} and used in \cite{han2025structure}.

\begin{definition}[$R$-long strainers]
    Given $k\in \mathbb{N}$ and $\delta,\alpha,R>0$, a $(k,\delta,\alpha)$-strainer $(p_1,\ldots,p_k)$ at $x$ with opposite strainer $(q_1,\ldots,q_k)$ is said to be \emph{$R$-long} if
    \begin{equation}
      \min_{i=1,\ldots,k}\{|p_ix|,|q_ix|\}\geq\delta^{-1}R.
    \end{equation}
    We denote by $\mathcal{A}(k,\delta,\alpha; R)\subset \mathcal{A}(k,\delta,\alpha)$ the set of all $(k,\delta,\alpha)$-strained points which admit an $R$-long $(k,\delta,\alpha)$-strainer.
\end{definition}

We now state the lemma.
Its proof is almost identical to that of \cite[Lemma 7.15]{han2025structure}.
The only substantive difference occurs in Step~2: the distance-angle estimates
coming from the geometry of cylindrical regions first give the almost
orthogonality of one comparison angle, and the weak quadruple condition is then
used, in place of the local semi-convexity assumption, to recover the remaining comparison angle estimates between strainer pairs.
Recall that for $E\subset X$, the quantity $\beta_E(r)$ denotes the largest possible cardinality of a maximal $r$-separated subset of $E$.

\begin{lemma}\label{lemma:singular_set_estimate}
    Let $X$ be an $n$-dimensional $S$-concave Busemann concave space with $S\geq 1$.
    Fix $k\in \{1,\ldots,n-1\}$, and let $\tilde{\varepsilon}\in [0,\delta_{k})$ and $\tilde{\delta}\in (0,\pi]$ be such that $X$ satisfies the $(\tilde{\varepsilon}, \tilde{\delta})$-weak quadruple condition.
    Suppose that $(p_1,\ldots, p_k)$ is an $R$-long $(k,\delta,\alpha)$-strainer on a neighborhood $U$ of $x\in X$ with common opposite strainer $(q_1,\ldots,q_k)$, where $0<\delta< \min\{\tilde{\delta},1/2\},\ \alpha>0$.
    Then there exists $r_x>0$ such that $B(x,r_x)\subset U$ and, for every $0<r<\delta r_x$ and every $m=(m_1,\ldots,m_k)\in \mathbb{Z}^k$, one has
    \begin{equation}\label{eq:singular_set_estimate_00}
        \beta_{D_x(r,m)\setminus \mathcal{A}\left(k+1,\delta,\max\{\alpha, 3\delta + \tilde{\varepsilon}\};\, r \right)}\left(\delta^{-1}r \right)
        \leq
        K(N,\delta,S),
    \end{equation}
    where $D_x(r,m)$ is an $r$-cylindrical region in $B(x,r_x)$ defined as
    \begin{equation}\label{eq:cylindrical_region}
        D_x(r,m):=\{z\in B(x,r_x)\colon 0.1(m_i-1)r\leq |p_ix|-|p_iz|\leq 0.1m_i r,\ i=1,\ldots,k\}.
    \end{equation}
    The constant $K$ depends only on the doubling constant $N$ of $X$, $\delta$, and $S$.
\end{lemma}
\begin{proof}
  Fix $k\in \{1,\ldots,n-1\}$, $R>0$, and $0<\delta<\min\{\tilde{\delta},1/2\}$.
  We choose $r_x>0$ as follows.
  First choose $r_0>0$ small enough such that $B(x,r_0)\subset U$, $r_0\leq R/2$ and
  \begin{equation}\label{eq:singular_set_estimate_1}
    \max\{4r_0/|q_kx|, \delta_S(4r_0;|q_kx|) \}<0.01\delta.
  \end{equation} 
  Set $r_x:=r_0/2\leq R/4$.
  By this choice of $r_x$, for any $y,z\in B(x,r_x)$ we have
  \begin{equation}
    |yz|\leq 0.01\delta\min\{|p_iy|,|q_iy|\},\quad i=1,\ldots,k.
  \end{equation}

  Fix $r\in (0,\delta r_x)$ and $m\in \mathbb{Z}^k$.
  For notational simplicity, set
  \begin{equation*}
    \mathcal{A}:=
    \mathcal{A}(k+1,\delta, \max\{\alpha, 3\delta + \tilde{\varepsilon}\};\, r).
  \end{equation*}
  We will prove, by contradiction, that the largest possible cardinality of a maximal $\delta^{-1}r$-separated subset of $D_x(r,m)\setminus \mathcal{A}$ is bounded above by
  \begin{equation}
    K(N,\delta,S):=\bar{K}(N,L_0,M)-1.
  \end{equation}
  Here $\bar{K}$ is the constant from Lemma~\ref{lemma:number_points}, $N$ is the doubling constant of $X$ implied by Proposition~\ref{prop:MCP}, $L_0:=L_0(\delta/2,S)$ is the constant from Lemma~\ref{lemma:quantitive_discrepancy_angles}, and $M:=M(\delta)= N_0(\beta(\delta/2))+2$, where $N_0$ is the constant from Lemma~\ref{lem:uniform_compactness_space_of_directions_common_length} depending on the doubling constant $N$, and $\beta$ is defined in \eqref{def:beta_delta_num}.
  
  Suppose that the claim above fails.
  Then there exists a maximal $\delta^{-1}r$-separated subset of $D_x(r,m)\setminus \mathcal{A}$ with at least $\bar{K}(N,L_0,M)$ elements.
  By Lemma~\ref{lemma:number_points}, we can choose $M$ points $\{x_i\}_{i=0}^{M-1}$ in this maximal $\delta^{-1}r$-separated subset such that $|x_{i+1}x_0|\geq L_0|x_ix_0|$ for $i=1,\ldots,M-2$.
  By Lemma~\ref{lem:uniform_compactness_space_of_directions_common_length} and the choice of $M=N_0(\beta(\delta/2))+2$, there exist two indices $1\leq i<j\leq M-1$ such that
  \begin{equation}
    \angle_{x_0}\left((\bar{\xi}_i,l), (\bar{\xi}_j,l)\right)
    =
    \angle_{x_0}\left( \xi_i(l), \xi_j(l) \right) < \beta\left(\delta/2\right),
  \end{equation}
  where $\xi_i,\xi_j$ are unit-speed geodesics from $x_0$ to $x_i$ and $x_j$, respectively, $\bar{\xi}_i,\bar{\xi}_j$ are their maximal extensions, and $l>0$ can be chosen arbitrarily, thanks to the positive scaling-invariance of angles of fixed scales; see Section~\ref{subsect:notions_of_angles}.
  Lemma~\ref{lemma:quantitive_discrepancy_angles} then implies that $x_j$ is a $(1,\delta)$-strainer at $x_i$ with opposite strainer $x_0$.
  Moreover, $(x_j,x_0)$ is an $r$-long strainer pair, since $x_0,x_i,x_j$ belong to a $\delta^{-1}r$-separated subset of $D_x(r,m)\setminus \mathcal{A}$.
  
  In the following, we show that
  \begin{equation*}
    \pmb{p}:=(p_1,\ldots,p_k,x_j)
  \end{equation*}
  is an $r$-long $(k+1, \delta, \max\{\alpha, 3\delta + \tilde{\varepsilon}\})$-strainer at $x_i$ with the opposite strainer
  \begin{equation*}
    \pmb{q}:=(q_1,\ldots,q_k,x_0).
  \end{equation*}
  This contradicts the choice of $x_i$, since $x_i\notin \mathcal{A}(k+1,\delta,\max\{\alpha, 3\delta + \tilde{\varepsilon}\};\, r)$.

  \begin{enumerate}[fullwidth, label=\textbf{Step \arabic*:}]
    \item In this step, we verify the first two conditions in Definition~\ref{def:k_strainer}.
    First, $(p_1,\ldots,p_k)$ is clearly an $r$-long $(k,\delta,\alpha)$-strainer at $x_i$, since it is an $R$-long $(k,\delta,\alpha)$-strainer on $U$ by the assumption of the lemma, and $r<R$.
    Second, we have already shown that $x_j$ is an $r$-long $(1,\delta)$-strainer at $x_i$ with opposite strainer $x_0$.
    Finally, by the choice of $r_x$ and \eqref{eq:singular_set_estimate_1}, it follows that $|q_kx_i|> r_0>|x_ix_j|$ and
    \begin{equation}
      \delta_S(|x_jx_i|;|q_kx_i|)
      \leq
      \delta_S(r_0; |q_kx|-r_0)
      \leq
      \delta_S(2r_0; |q_kx|)
      <\delta.
    \end{equation}
    Thus the first two conditions in Definition~\ref{def:k_strainer} are satisfied for $\pmb{p}$ at $x_i$ with opposite strainer $\pmb{q}$.

    \item It remains to verify the almost orthogonality condition in Definition~\ref{def:k_strainer} for the comparison angles $\tilde{\angle}p_lx_ix_j, \tilde{\angle}q_lx_ix_j, \tilde{\angle}p_lx_ix_0$ and $\tilde{\angle}q_lx_ix_0$ for $l=1,\ldots,k$.
    For notational simplicity, fix $l\in \{1,\ldots, k\}$ and denote $p:=p_l$ and $q:=q_l$.
    
    We first verify the almost orthogonality of $\tilde{\angle}px_ix_j$.
    Applying the Euclidean law of cosines to the comparison triangle $\tilde{\Delta}px_ix_j$, we obtain
    \begin{align}\label{eq:prop_H_dim_1}
      \cos\tilde{\angle}px_ix_j
      &=
      \frac{|px_i|^2+|x_ix_j|^2-|px_j|^2}{2|px_i||x_ix_j|}\nonumber\\
      &=
      \frac{\left(|px_i|+|px_j|\right)\left(|px_i|-|px_j|\right)}{2|px_i||x_ix_j|}
      +
      \frac{|x_ix_j|}{2|px_i|}.
    \end{align}
    Following exactly the same argument as in (7.32) and (7.33) of Step 2 in \cite[Lemma 7.15]{han2025structure}, we obtain
    \begin{equation}
        \left|\cos\tilde{\angle}px_ix_j \right|\leq \delta/2.
    \end{equation}
    We note that this argument uses only the geometry of the cylindrical region $D_x(r,m)$ and the fact that $x_i,x_j,x_0\in D_x(r,m)\setminus \mathcal{A}$, the distance assumption $|x_ix_j|\geq \delta^{-1}r$, and the relationship among $r, r_0$ and $R$; it does not rely on local semi-convexity.
    Consequently,
    \begin{equation}
        |\tilde{\angle}px_ix_j - \pi/2|< \delta.
    \end{equation}
    
    We next consider the comparison angle $\tilde{\angle}qx_ix_j$.
    By the choice of $r_x$, we have
    \begin{equation}
      |x_ix_j|\leq 0.01\delta \min\{|px_i|,|px_j|, |qx_i|,|qx_j|\}.
    \end{equation}
    This implies that
    \begin{equation}\label{eq:singular_set_estimate_3}
      \tilde{\angle}px_i x_j + \tilde{\angle}px_jx_i>\pi-\delta/2,\quad
      \tilde{\angle}qx_ix_j + \tilde{\angle}qx_jx_i>\pi-\delta/2. 
    \end{equation}
    Combining the first inequality in \eqref{eq:singular_set_estimate_3} with $|\tilde{\angle}px_ix_j-\pi/2|< \delta$, we get
    \begin{equation}
      \tilde{\angle}px_jx_i > \pi - \delta/2 - (\pi/2 + \delta)= \pi/2 - 3\delta/2.
    \end{equation}
    Applying the $(\tilde{\varepsilon}, \tilde{\delta})$-weak quadruple condition to the quadruple $(x_j; p,q,x_i)$ gives
    \begin{multline}
      \tilde{\angle}qx_jx_i
      \leq
      2\pi + \tilde{\varepsilon} - \tilde{\angle}px_jq - \tilde{\angle}px_jx_i
      <
      2\pi + \tilde{\varepsilon} -(\pi - \delta) - (\pi/2 - 3\delta/2)\\
      =
      \pi/2 + \tilde{\varepsilon} + 5\delta/2.
    \end{multline}
    Substituting this into the second inequality in \eqref{eq:singular_set_estimate_3}, we obtain
    \begin{equation}
      \tilde{\angle}qx_i x_j
      >
      \pi - \delta/2 - \tilde{\angle}qx_jx_i
      >
      \pi/2  - \tilde{\varepsilon} - 3\delta.
    \end{equation}
    Similarly, by applying the same argument to the quadruple $(x_i; p,q,x_0)$, we also obtain $|\tilde{\angle}px_ix_0-\pi/2|<\delta$ and $\tilde{\angle}qx_ix_0> \pi/2 - \tilde{\varepsilon} - 3\delta$.
  \end{enumerate}

  Combining these estimates, we conclude that $\pmb{p}$ is an $r$-long $(k+1,\delta,\max\{\alpha, 3\delta + \tilde{\varepsilon}\})$-strainer at $x_i$ with the opposite strainer $\pmb{q}$, which is a contradiction.
  Therefore, our claim follows, and \eqref{eq:singular_set_estimate_00} holds.
\end{proof}

We now apply this lemma to estimate the Hausdorff dimension of singular sets.

\begin{theorem}[Hausdorff dimension of singular sets]\label{thm:singular_set_dimension}
    Let $X$ be an $n$-dimensional $S$-concave Busemann concave space with $S\geq 1$.
    Then, for every $0<\delta<1/2$, the set $X\setminus \mathcal{A}(1,\delta)$ is discrete, and $\dim_{H}(X\setminus \mathcal{A}(k,\delta,\delta))\leq k-1$ for $k=1,\ldots,n$. 
\end{theorem}
\begin{proof}
  We first show that $X\setminus \mathcal{A}(1,\delta)$ is discrete.
  In fact, we show that the cardinality of $X\setminus \mathcal{A}(1,\delta)$ is at most $K(N,\delta,S)$, where $K$ is the constant in Lemma~\ref{lemma:singular_set_estimate}.
  Suppose, by contradiction, that this bound fails.
  Then $X\setminus \mathcal{A}(1,\delta)$ contains at least $K(N,\delta,S)+1=\bar{K}(N,L_0,M)$ points.
  Following the same argument as in Step 1 of the proof of Lemma~\ref{lemma:singular_set_estimate}, we can find three points $x_0,x_i,x_j\in X\setminus \mathcal{A}(1,\delta)$ such that $x_j$ is a $(1,\delta)$-strainer at $x_i$ with the opposite strainer $x_0$.
  This contradicts $x_i\notin \mathcal{A}(1,\delta)$.
  Hence the cardinality of $X\setminus \mathcal{A}(1,\delta)$ is at most $K(N,\delta,S)$.
  In particular, $X\setminus \mathcal{A}(1,\delta)$ is discrete and $\dim_H(X\setminus \mathcal{A}(1,\delta))=0$.

  We now prove the second assertion.
  Let $0<\varepsilon<\delta_n/3$, and let $\tilde{\delta}:=\tilde{\delta}(\varepsilon,S)$ be the constant given by Lemma~\ref{lemma:verification_WQC_II} such that $X$ satisfies the $(\varepsilon,\delta)$-weak quadruple condition for all $\delta\in (0,\tilde{\delta}]$.
  Fix $0<\delta<\min\{\tilde{\delta},1/2\}$ sufficiently small so that $3\delta+\varepsilon<1/2$.
  We prove by induction that
  \begin{equation}
    \dim_{H}(X\setminus \mathcal{A}(k,\delta,3\delta + \varepsilon))\leq k-1,
  \end{equation} 
  for $k=1,\ldots,n$.
  The case $k=1$ follows directly from the first assertion.
  Suppose that $\dim_{H}(X\setminus \mathcal{A}(k,\delta, 3\delta + \varepsilon))\leq k-1$ for some $k\in \{1,\ldots,n-1\}$.

  For any $x\in \mathcal{A}(k,\delta,3\delta + \varepsilon)$, we can find $R_x>0$ such that $x$ admits a $4R_x$-long $(k,\delta, 3\delta+\varepsilon)$-strainer $\pmb{p}$.
  By Proposition~\ref{prop:almost_orthogonality}, after possibly shrinking the neighborhood, there exists an open neighborhood $U_x$ of $x$ such that $\pmb{p}$ is an $R_x$-long $(k,\delta, 3\delta+\varepsilon)$-strainer on $U_x$ with common opposite strainer $\pmb{q}$.
  Let $r_x>0$ be the radius given by Lemma~\ref{lemma:singular_set_estimate}.

  We claim that the Hausdorff dimension of $B(x,r_x)\setminus \mathcal{A}(k+1,\delta, 3\delta + \varepsilon)$ is at most $k$.
  Indeed, it is straightforward to verify that $B(x,r_x)$ can be covered by the cylindrical regions $\{D_x(r,m)\}_{m\in \mathbb{Z}^k}$, and that the number of nonempty such regions is bounded above by $(2r_x/(0.1r))^k$.
  Applying Lemma~\ref{lemma:singular_set_estimate} with $\alpha:=3\delta + \varepsilon$, we obtain that, for every $0<r<\delta r_x$,
  \begin{multline}\label{eq:singular_set_dimension_1}
    \beta_{B(x,r_x)\setminus \mathcal{A}(k+1, \delta, 3\delta + \varepsilon)}\left(\delta^{-1}r\right)
    \leq
    \sum_{\substack{m\in \mathbb{Z}^k\\D_x(r,m)\neq \emptyset}}\beta_{D_x(r,m)\setminus \mathcal{A}(k+1, \delta, 3\delta + \varepsilon)}\left(\delta^{-1}r\right)\\
    \leq
    \sum_{\substack{m\in \mathbb{Z}^k\\D_x(r,m)\neq \emptyset}}\beta_{D_x(r,m)\setminus \mathcal{A}(k+1, \delta, 3\delta + \varepsilon; r)}\left(\delta^{-1}r\right)
    \leq
    \left(\frac{2r_x}{0.1r}\right)^k K(N,\delta,S).
  \end{multline}
  Hence, for every $a>0$, it follows that
  \begin{equation}
    \limsup_{r\to 0}\left(\delta^{-1}r\right)^{k+a}\beta_{B(x,r_x)\setminus \mathcal{A}(k+1, \delta, 3\delta + \varepsilon)}\left(\delta^{-1}r\right)
    =
    0.
  \end{equation}
  Thus the rough dimension of $B(x,r_x)\setminus \mathcal{A}(k+1, \delta, 3\delta + \varepsilon)$ is at most $k$; see Section~\ref{subsect:basics}. 
  Consequently, by \eqref{eq:relation_different_dims}, it follows that $\dim_H(B(x,r_x)\setminus \mathcal{A}(k+1, \delta, 3\delta + \varepsilon))\leq k$.

  Finally, since $X$ is proper, we may cover $\mathcal{A}(k,\delta, 3\delta + \varepsilon)$ by at most countably many such balls $B(x,r_x)$.
  Then it follows that
  \begin{equation}
    \dim_H\left(\mathcal{A}(k,\delta, 3\delta + \varepsilon) \setminus \mathcal{A}(k+1, \delta, 3\delta + \varepsilon)\right)
    \leq
    k.
  \end{equation}
  Combining this with the induction hypothesis, we obtain
  \begin{multline}
    \dim_H\left(X\setminus \mathcal{A}(k+1, \delta, 3\delta + \varepsilon)\right)\\
    =
    \max\left\{\dim_H\left(X\setminus \mathcal{A}(k,\delta, 3\delta + \varepsilon)\right), \dim_{H}\left(\mathcal{A}(k,\delta, 3\delta + \varepsilon)\setminus \mathcal{A}(k+1, \delta, 3\delta + \varepsilon) \right)\right\}\\
    \leq
    k.
  \end{multline}
  This closes the induction.
  Given $0<\delta'<1/2$, choose $0<\varepsilon<\delta_n/3$ and then choose $0<\delta<\min\{\tilde{\delta}(\varepsilon,S),1/2\}$ sufficiently small so that
  \[
    \delta<\delta',
    \qquad
    3\delta+\varepsilon\leq \delta'.
  \]
  Then
  \[
    \mathcal{A}(k,\delta,3\delta+\varepsilon)\subset \mathcal{A}(k,\delta',\delta'),
  \]
  and hence
  \[
    X\setminus \mathcal{A}(k,\delta',\delta')
    \subset
    X\setminus \mathcal{A}(k,\delta,3\delta+\varepsilon).
  \]
  The desired estimate follows from the preceding bound.
\end{proof}

As a direct consequence of Theorem~\ref{thm:singular_set_dimension}, we obtain the following measure-theoretic stratification of finite-dimensional $S$-concave Busemann concave spaces.
\begin{corollary}[Measure-theoretic stratification]\label{cor:measure_stratification}
    Let $X$ be an $n$-dimensional $S$-concave Busemann concave space with $S\geq 1$.
    Then $X$ admits a measure-theoretic stratification $\{X_k\}_{k=0}^n$ such that $X$ is the disjoint union of the sets $X_1\ldots,X_n$, and $\dim_{H}(X_k)\leq k$ for each $k=0,\ldots,n$.
    Moreover, the top-dimensional stratum $X_n$ is an open dense topological $n$-manifold, and the lowest-dimensional stratum $X_0$ is discrete.
\end{corollary}
\begin{proof}
  First, observe that $\mathcal{A}(n+1, \delta, \delta)=\emptyset$ for every $0<\delta<1/2$.
  Otherwise, the self-improvement property of strainers (Lemma~\ref{lemma:almost_self_improvement}) would imply that the strainer number of $X$ is at least $n+1$, contradicting the assumption that $\dim_H(X)= n$.

  Now choose $0<\delta<1/2$ sufficiently small as in Theorem~\ref{thm:top_manifold_part}, so that $\mathcal{A}(n,\delta,\delta)$ is an open dense topological $n$-manifold.
  Set
  \[
    X_0:=X\setminus \mathcal{A}(1,\delta,\delta),
  \]
  and, for $1\leq k\leq n-1$,
  \[
    X_k:=\mathcal{A}(k,\delta,\delta)\setminus \mathcal{A}(k+1,\delta,\delta),
    \qquad
    X_n:=\mathcal{A}(n,\delta,\delta).
  \]
  Then $X$ decomposes into the disjoint union of the sets $X_0,\ldots,X_n$.
  The remaining assertions follow directly from Theorem~\ref{thm:singular_set_dimension}.
\end{proof}


\medskip

\section{Further problems and discussion}\label{sect:problems}
\noindent
In this final section, we record some problems concerning the weak quadruple condition and finer geometric and topological structures of $S$-concave Busemann concave spaces.
These questions will be investigated in future work.
For related problems concerning Busemann convex spaces, we refer to \cite{fujioka2025top,fujioka2026busemann}.

The first problem concerns the characterization of the weak quadruple condition introduced in this paper.
It is well-known that Alexandrov spaces with curvature bounded below can equivalently be described by the classical quadruple condition.
This condition, in turn, admits several equivalent characterizations, including the so-called \emph{Lang--Schroeder--Sturm inequality}.
Such an inequality was first discovered by Lang--Schroeder \cite{lang1997kirszbraun} and later rediscovered by Sturm \cite{sturm1999metric} (see also \cite{yokota2012rigidity,alexander2024alexandrov}).
On the other hand, the classical quadruple condition can also be characterized by isometric embeddings of quadruples into model spaces of constant curvature, as shown by Berestovskii \cite{berestovskii1986spaces}, following Wald \cite{wald1935begrundung}; see \cite{plaut1996spaces, lang1997kirszbraun} for more details.
Motivated by these characterizations, we ask:

\begin{problem}
  Is there an analogous discrete characterization of $S$-concave Busemann concave spaces, or of the $(\varepsilon,\delta)$-weak quadruple condition, in terms of mutual distances of finite sequences of points similar to the Lang--Schroeder--Sturm inequality?
  Separately, is there a characterization in terms of isometric embeddings of quadruples into suitable model Banach spaces?
\end{problem}

We mention that discrete characterizations by finite pairwise distance data also appear in the theory of Banach spaces, particularly in the study of \emph{roundness}; see, for example, \cite{amini2021roundness} and references therein.
It is unclear to us whether these characterizations are related to the problem posed here.

The next problem asks about curvature-dimension conditions in this setting, beyond the measure contraction property established in Proposition~\ref{prop:MCP}.
\begin{problem}
  Does an $n$-dimensional $S$-concave Busemann concave space, equipped with the $n$-dimensional Hausdorff measure, satisfy the $\mathrm{CD}(0,n)$ condition?
  More generally, if a Busemann concave space admits a nontrivial $n$-dimensional Hausdorff measure, does it satisfy the $\mathrm{CD}(0,n)$ condition?
\end{problem}
It is well-known that finite-dimensional Alexandrov spaces with curvature bounded below satisfy the $\mathrm{CD}$ condition.
This compatibility was first shown by Petrunin \cite{petrunin2010alexandrov} for Alexandrov spaces with non-negative curvature, and later extended by Zhang--Zhu \cite{zhang2009ricci} to the general case of Alexandrov spaces with curvature bounded below.
On the other hand, it has been shown by Ohta \cite{ohta2009finsler} that the $\mathrm{CD}$ condition extends naturally to the Finsler setting.
In particular, any connected, forward geodesically complete, smooth $n$-dimensional Finsler manifold of Berwald type, with non-negative flag curvature and equipped with the Busemann--Hausdorff measure, satisfies the $\mathrm{CD}(0,n)$ condition; see \cite[Theorem 1.2 and Remark 8.2(b)]{ohta2009finsler}.

The next problem concerns the rigidity of tangent cones of $S$-concave Busemann concave spaces.
\begin{problem}
  Let $X$ be an $n$-dimensional $S$-concave Busemann concave space with $S\geq 1$.
  Does $X$ admit a `Berwald-type' structure, in the sense that all of its Banach tangent cones are isometric to the same Banach space?
  Furthermore, do any two interior points of a geodesic in $X$ have isometric tangent cones?  
\end{problem}
Regarding the first question, we recall that tangent cones of Finsler manifolds generally vary from point to point. 
However, Lytchak--Ivanov \cite{ivanov2019rigidity} shows that any Finsler manifold satisfying the Busemann NPC condition is necessarily a Berwald manifold with non-positive flag curvature.
In particular, a key consequence in \cite[Lemma 3.1]{ivanov2019rigidity} suggests that Busemann convexity implies strong control over the variation of tangent norms along geodesics.
Nevertheless, it remains unclear whether the first question has an affirmative answer even for general Busemann NPC spaces (see \cite[Problem 8.4]{fujioka2025top}).
The corresponding problem for $S$-concave Busemann concave spaces appears even more subtle.
Regarding the second question, the corresponding question for finite-dimensional Alexandrov spaces with curvature bounded below was conjectured by Burago--Gromov--Perelman \cite[Remark 7.18]{burago1992ad} and subsequently proved by Petrunin \cite{petrunin1998parallel}.

The next problem concerns the existence of a continuous Finsler structure on finite-dimensional $S$-concave Busemann concave spaces, motivated by the work of Otsu--Shioya \cite{otsu1994riemannian} on finite-dimensional Alexandrov spaces with curvature bounded below.

\begin{problem}\label{problem:Finsler_structure}
  Let $X$ be an $n$-dimensional $S$-concave Busemann concave space with $S\geq 1$.
  Does $X$ admit a $C^0$-Finsler structure on $X$ outside some singular set such that the metric induced by this Finsler structure coincides with the original metric of $X$?
\end{problem}
In the Alexandrov setting, Otsu--Shioya \cite{otsu1994riemannian} proved an even stronger result: they established not only the existence of a continuous Riemannian structure on the regular part of a finite-dimensional Alexandrov space with curvature bounded below, but also the $C^{1/2}$-continuity of this Riemannian metric on a certain subset of the regular part.
On the other hand, in the setting of Busemann $\mathrm{G}$-spaces, Pogorelov \cite{pogorelov1998busemann} obtained a continuous Finsler structure under some additional regularity assumption, called \emph{Axiom A}, concerning the continuous differentiability of distance functions.
More recently, Fujioka--Gu \cite{fujioka2026finsler} studied Busemann $\mathrm{G}$-spaces satisfying local semi-concavity or semi-convexity, and obtained Finsler structures of higher regularity for these spaces.
However, the local geodesic extension property and the continuity of angles play important roles in the results for Busemann $\mathrm{G}$-spaces, whereas neither property is available for $S$-concave Busemann concave spaces.

Our final two problems concern the finer geometric and topological properties of finite-dimensional $S$-concave Busemann concave spaces.
\begin{problem}
  Let $X$ be an $n$-dimensional $S$-concave Busemann concave space with $S\geq 1$.
  Does the set of interior singular points of $X$, in a suitable sense, have Hausdorff codimension at least $2$?
  Does $X$ admit a canonical stratification whose strata are topological manifolds?
\end{problem}

It is unclear to us what the appropriate definition of interior points should be used.
In the Alexandrov setting, interior points are understood as non-boundary points, where the boundary is defined inductively based on the metric structure of tangent cones; see \cite[Definition 7.19]{burago1992ad}.
This definition does not directly extend to our setting, since tangent cones of $S$-concave Busemann concave spaces need not be metric cones.

On the other hand, known notions of interior points seem too restrictive for the present purpose. 
For example, \emph{geometrically inner points} in the sense of Lytchak--Schroeder \cite[Definition 1.4]{lytschak2004affine} (see also \cite[Section 3]{fujioka2026busemann}) are regular in the Alexandrov setting, as follows from the maximal radius theorem; see \cite[Corollary 16.25]{alexander2024alexandrov} and \cite{grove2022alexandrov}.
However, Alexandrov spaces may have many interior points that are not regular.
Regarding stratification, we recall that Perelman \cite{perel1994elements} established such a topological stratification for finite-dimensional Alexandrov spaces with curvature bounded below by proving, via Morse theory for distance functions, that they are $\mathrm{MCS}$-spaces\footnote{Here $\mathrm{MCS}$ stands for \emph{multiple conic singularities.}}.
It is unclear to us whether this approach can be adapted to the setting of $S$-concave Busemann concave spaces.

\begin{problem}\label{pbm:two_sides}
  Let $X$ be an $n$-dimensional $S$-concave Busemann concave space with $S\geq 1$, which is further locally semi-convex in the sense of \cite{han2025structure}.
  Is $X$ a topological $n$-manifold, possibly with boundary?
  Does it admit local charts with higher regularity?
\end{problem}
We mention that Fujioka--Gu \cite{fujioka2026finsler} obtained related results under the additional assumption that the space is a $\mathrm{G}$-space.
In particular, they showed that such spaces are topological manifolds and admit $C^{1,1/2}$-atlases; see \cite[Theorem~1.1 and~1.2]{fujioka2026finsler} for precise statements.
In the Alexandrov setting, spaces with two-sided Alexandrov bounds have been previously studied by Alexandrov, Nikolaev and Berestovskii; see \cite{berestovskij1993multidimensional}.
See also \cite{kapovitch2020cd,kapovitch2022structure} for structural results on $\mathrm{CD}$ spaces with Alexandrov curvature bounded above, and the recent work of Fujioka--Tashiro \cite{fujioka2026busemann} on Busemann NPC spaces with synthetic lower curvature bounds in the sense of the measure contraction property.
These works provide related evidence that combining upper and lower synthetic curvature bounds may lead to stronger manifold regularity.

\bigskip

\sectionnotoc{Acknowledgments}
\noindent
The authors would like to express their sincere gratitude to Professors Tadashi Fujioka and Shijie Gu for helpful discussions on the finer structures of Busemann and $\mathrm{G}$-spaces, weak quadruple comparison, and for bringing Problem~\ref{pbm:two_sides} and the reference \cite{pogorelov1998busemann} to their attention.
They are also grateful to Professor Nan Li for helpful discussions on the possible extension of Alexandrov's lemma to the Busemann setting. 
The authors are also indebted to Professors Alexander Lytchak and Shin-ichi Ohta for their interest in this work, their helpful comments and stimulating discussions, and for pointing out the relevance of the structure theory of sets with positive reach.
Part of this work was developed during the authors' stay in Fukuoka for the conference `The 18th MSJ-SI: Analysis, Geometry and Probability on Metric Measure Spaces'.
\bigskip

\sectionnotoc{Funding}
\noindent
This work is supported in part by the National Natural Science Foundation of China (Nos.~12201596 and 12431004), the Shandong Provincial Natural Science Foundation (ZR2025QB05), the Taishan Scholars Program of Shandong Province (tsqn202408059), and the Scientific Research Innovation Capability Support
Project for Young Faculty (SRICSPYF-ZY2025160).

\bigskip

\sectionnotoc{Declaration}
\noindent
The authors declare that there is no conflict of interest and this paper has no associated data.
During the preparation of this manuscript, the authors used ChatGPT (OpenAI) only for partial language polishing. 
All mathematical arguments, results and proofs were thoroughly produced, reviewed and verified by the authors.

\bigskip
\bigskip

\begin{flushleft}
\small \normalfont
\textsc{Bang-Xian Han\\
School of Mathematics, Shandong University, Jinan, 250100, China}\\
\texttt{\textbf{hanbx@sdu.edu.cn}}
\end{flushleft}

\medskip
\begin{flushleft}
\small \normalfont
\textsc{Liming Yin\\
School of Mathematical Sciences, University of Science and Technology of China, Hefei, 230026, China}\\
\texttt{\textbf{yinliming@ustc.edu.cn}}
\end{flushleft}

\end{document}